\documentclass[final]{siamltex}

\usepackage[top=1.2in, bottom=1.2in, left=1in, right=1in]{geometry}
\usepackage{graphicx,amssymb}
\usepackage{booktabs}
\usepackage{array}
\usepackage{color}
\usepackage{multirow}
\usepackage{amsmath}
\newtheorem{assumption}{Assumption}[section]
\newtheorem{remark}{Remark}[section]
\newtheorem{example}{Example}[section]
\usepackage[pdfstartview=FitH,colorlinks,pdfborder=001,linkcolor=blue,anchorcolor=blue,citecolor=blue]{hyperref}
\usepackage{algorithm}  
\usepackage{algorithmic} 
\usepackage{caption}
\usepackage{subcaption}
\usepackage[misc]{ifsym} 
\numberwithin{equation}{section}

\title{A multi-level ADMM algorithm for elliptic PDE-constrained optimization problems
}

\author{Xiaotong Chen\thanks{School of Mathematical Sciences, Dalian University of Technology, Dalian, Liaoning 116025, China. ({\tt chenxiaotong@mail.dlut.edu.cn, yubo@dlut.edu.cn, chenzixuan@mail.dlut.edu.cn}).} 
\and Xiaoliang Song\thanks{Department of Applied Mathematics, The Hong Kong Polytechnic University, Hung Hom, Kowloon, Hong Kong, China. ({\tt xiaoliang.song@polyu.edu.hk}).}
\and Zixuan Chen\footnotemark[1]
\and Bo Yu\footnotemark[1]
}

\begin{document}

\maketitle

\begin{abstract}
In this paper, the elliptic PDE-constrained optimization problem with box constraints on the control is studied. To numerically solve the problem, we apply the \emph{`optimize-discretize-optimize'} strategy. Specifically, the alternating direction method of multipliers (ADMM) algorithm is applied in function space first, then the standard piecewise linear finite element approach is employed to discretize the subproblems in each iteration. Finally, some efficient numerical methods are applied to solve the discretized subproblems based on their structures. Motivated by the idea of the multi-level strategy, instead of fixing the mesh size before the computation process, we propose the strategy of gradually refining the grid. Moreover, the subproblems in each iteration are solved inexactly. Based on the strategies above, an efficient convergent multi-level ADMM (mADMM) algorithm is proposed. We present the convergence analysis and the iteration complexity results $o(1/k)$ of the proposed algorithm for the PDE-constrained optimization problems. Numerical results show the high efficiency of the mADMM algorithm. 
\end{abstract}

\begin{keywords}
PDE-constrained optimization, ADMM, multi-level, convergence analysis
\end{keywords}

\begin{AMS}
49N05, 49M25, 65N12, 68W15
\end{AMS}


\section{Introduction}
\label{intro}
In this paper, we consider the elliptic PDE-constrained optimization  problem with box constraints on the control:
\begin{equation}\label{eqn:orginal problems}
\begin{aligned}
\min \limits_{(y,u)\in Y\times U}^{}\ \ &J(y,u)=\frac{1}{2}\|y-y_d\|_{L^2(\Omega)}^{2}+\frac{\alpha}{2}\|u\|_{L^2(\Omega)}^{2} \\
{\rm s.t.}\ \ \ \ \ \ &Ly=u+y_{r}\ \ \ \ \ \mathrm{in} \  \Omega, \\
&\ \ y=0 \quad  \ \ \ \ \ \ \ \  \ \mathrm{on} \ \partial\Omega,\\
\qquad \qquad\quad &\ \ u\in  U_{ad}=\{v(x)|a\leq v(x)\leq b, {\rm a.e }\  \mathrm{on}\ \Omega\}\subseteq U,
\end{aligned} \tag{$\mathrm{P}$}
\end{equation}
where $Y:=H_{0}^{1}(\Omega), U:=L^{2}(\Omega), \Omega\subseteq\mathbb{R}^{n}(n=2,3)$ is a convex, open and bounded domain with $C^{1,1}$-  or polygonal boundary; the desired state $y_{d}\in L^{2}(\Omega)$ and the source term $y_{r}\in L^{2}(\Omega)$ are given; parameters $\alpha>0$, $-\infty<a<b<+\infty$; $L$ is the uniformly elliptic differential operater given by
\begin{equation*}
Ly:=-\sum\limits_{i,j=1}^{n}(a_{ij}y_{x_{i}})_{{x_{j}}}+c_{0}y, 
\end{equation*}
where $a_{ij},c_{0}\in L^{\infty}(\Omega),c_{0}\geqslant 0,a_{ij}=a_{ji},\sum\limits_{i,j=1}^{n}a_{ij}(x)\xi_{i}\xi_{j}\geqslant\theta\|\xi\|^{2}, \ \rm a.a. \ x\in \Omega, \forall \xi\in \mathbb{R}^{n}.$

By the classical conclusion in the PDE theory, for a given $y_{r}\in L^{2}(\Omega)$ and every $u\in L^{2}(\Omega)$, elliptic PDEs involved in the (\ref{eqn:orginal problems})
\begin{equation}\label{eqn:pde}
\begin{aligned}
Ly&=u+y_{r}\quad   \mathrm{in}\ \Omega, \\
y&=0\quad \quad \quad \ \mathrm{on}\ \partial\Omega,\\
\end{aligned}
\end{equation}
has a unique weak solution $y=y(u):=S(u+y_{r})$, where $S:L^{2}(\Omega)\rightarrow H_{0}^{1}(\Omega)$ denotes the solution operator. It is well-defined and is a continuous linear injective operator\cite[Theorem B.4]{Kinderlehrer1980An}.

As is known to all, there are two possible approaches to tackle optimization problems with PDE constraints numerically. One is \emph{First discretize, then optimize}, another is \emph{First optimize, then discretize} \cite{Hinze2009Optimization}. In the first approach, the discretization is applied to the original PDE-constrained optimization  problem, while in the second one, the discretization is applied to the KKT system of the PDE-constrained optimization  problem. 
Different from the strategies mentioned above, instead of applying discretized concept to problem (\ref{eqn:orginal problems}) or its KKT system directly, \emph{`optimize-discretize-optimize'} strategy was proposed by Song in \cite{Song2018}. First the optimization algorithm is given in the sense of continuous function space, then the subproblems in each iteration are discretized by the standard piecewise linear finite element approach. Finally, numerical optimization methods are applied to solve the discretized subproblems numerically. The advantage of this method is that from the optimization algorithm in function space, we can have a better knowledge of the structure of the PDE-constrained optimization  problem, which is important for choosing an appropriate discretization format to discretize the subproblems. Moreover, this strategy gives the freedom to discretize the subproblems differently, which makes the subproblems can be solved more effectively.  

In this paper, we focus on the alternating direction method of multipliers (ADMM) method, which was originally proposed by Glowinski et al. and Gabay et al. in \cite{Gabay1976A,Glowinski1975} and has been broadly used in many areas. Motivated by the success of ADMM in solving finite dimensional large scale optimization problem \cite{Chen2017,Fazel2013,Li2015,Li2016}, ADMM-type method has been used to solve PDE-constrained optimization  problems in function space in \cite{Chen2018,li2018efficient,Song2018A,zhang2017alternating}. While these works all focus on \emph{First discretize, then optimize} method, and to the best of our knowledge, very little work has been done to apply \emph{First optimize, then discretize} to solve PDE-constrained optimization  problems by ADMM-type method. We apply ADMM as the outer optimization algorithm of the \emph{`optimize-discretize-optimize'} strategy, propose a new algorithm to solve PDE-constrained optimization  problems efficiently and give the convergence analysis.

First,  we briefly give the iterative format of the classical ADMM for the 2-block convex optimization problem with linear constraints:
\begin{equation}
\begin{aligned} 
\min \limits_{(x,y)\in X\times Y}^{}&f(x)+g(y) \\
{\rm s.t.}\ \ \ \ &Ax+By=c,
\end{aligned}
\end{equation}
where $X$,$Y$,$\Lambda$ are finite dimensional real Euclidean spaces, $f(x):X\rightarrow (-\infty, +\infty]$, $g(y):Y\rightarrow (-\infty, +\infty]$ are convex functions (maybe nonsmooth), $A: X\rightarrow \Lambda$, $B: Y\rightarrow \Lambda$ are linear mappings, $c \in \Lambda$ is given.
The augmented Lagrangian function is given by the following form:
\begin{equation}
  \mathcal{L}_{\rho}(x,y,\lambda;\rho)=f(x)+g(y)+(\lambda,Ax+By-c)+\frac{\rho}{2}\|Ax+By-c\|^2,
\end{equation}
where $\lambda \in \Lambda$ denotes the Lagrange multiplier, $\rho>0$ is the penalty parameter.

Given $(x^0,y^0,\lambda^0)\in {\rm dom}f\times{\rm dom}g\times \Lambda$, penalty parameter $\rho>0$ and the step size parameter $\tau\in \left(0,\frac{\sqrt{5}+1}{2}\right)$, then the iterative format of the classical ADMM is as follows:
\begin{equation}
\left\{\begin{aligned}
x^{k+1}&={\text{argmin}}\mathcal{L}_{\rho}(x,y^k,\lambda^k;\rho),\\
y^{k+1}&={\text{argmin}} \mathcal{L}_{\rho}(x^{k+1},y,\lambda^k;\rho),\\
\lambda^{k+1}&=\lambda^k+\tau\rho(Ax^{k+1}+By^{k+1}-c).
\end{aligned} \right.
\end{equation}

The advantage of ADMM is that it separates $f(x)$ and $g(y)$ into two subproblems. In each subproblem, there is only one variable, the other one is fixed. Thus each subproblem could be solved easily and efficiently. For the classical ADMM, the convergence analysis was first conducted by \cite{Fortin1983Chapter,Gabay1976A,Glowinski1980Lectures}, and for the recent interesting new developments on the convergence analysis of ADMM-type method, see \cite{han2017linear,Sun2014A,2018arXiv180904249Y}. 

To apply ADMM-type method to solve the problem (\ref{eqn:orginal problems}), we use the solution operator $S$ and introduce an artificial variable $z$, then we equivalently rewrite problem (\ref{eqn:orginal problems}) as the following reduced form:
\begin{equation}\label{eqn:reduced problem with linear constraint}
\begin{aligned}
\min \limits_{(u,z)\in U\times Z}^{}\ \ \  &\hat{J}(u)+\delta_{U_{ad}}(z) \\
{\rm s.t.}\qquad \ &u=z, 
\end{aligned} \tag{RP}
\end{equation}
where $\hat{J}(u):=\dfrac{1}{2}\|S(u+y_{c})-y_{d}\|_{L^{2}(\Omega)}^{2}+\dfrac{\alpha}{2}\|u\|_{L^{2}(\Omega)}^{2}$ denotes the reduced cost function, $\delta_{U_{ad}}(z)$ denotes the indicator function of $U_{ad}$, 
\begin{equation}
\delta_{U_{ad}}(z)=\left\{\begin{array}{ll}{1,} & {z \in U_{ad}}, \\ {\infty,} & {z \notin U_{ad}}.\end{array}\right.
\end{equation}
The augmented Lagrangian function of (\ref{eqn:reduced problem with linear constraint}) is defined as follows:
\begin{equation}\label{eqn:lagrangian function}
L_{\sigma}(u,z,\lambda;\sigma)=\hat{J}(u)+\delta_{U_{ad}}(z)+\langle\lambda,u-z\rangle_{L^{2}(\Omega)}+\dfrac{\sigma}{2}\|u-z\|_{L^{2}(\Omega)}^{2},
\end{equation}
where $\lambda \in L^{2}(\Omega)$ denotes the Lagrangian multiplier, $\sigma>0$ is the penalty parameter. 

The classical ADMM in finite dimensional spaces can be extended directly in Hilbert space for problem (\ref{eqn:reduced problem with linear constraint}). Given $(u^{0},z^{0},{\lambda}^{0})\in L^{2}(\Omega)\times {\rm dom}(\delta_{U_{ad}}(\cdot))\times L^{2}(\Omega)$, parameters $\sigma>0$, $\tau \in \left(0,\dfrac{1+\sqrt{5}}{2}\right)$, the iterative scheme is presented as follows:

\begin{equation}
\left\{\begin{aligned}
\bar u^{k+1}&={\rm argmin} \hat{J}(u)+\langle \bar \lambda^{k},u-\bar z^{k}\rangle_{L^{2}(\Omega)}+\dfrac{\sigma}{2}\|u-\bar z^{k}\|_{L^{2}(\Omega)}^{2},\\
\bar z^{k+1}&={\rm argmin} \delta_{U_{ad}}(z)+\langle \bar \lambda^{k},\bar u^{k+1}-z\rangle_{L^{2}(\Omega)}+\dfrac{\sigma}{2}\|\bar u^{k+1}-z\|_{L^{2}(\Omega)}^{2},\\
\bar \lambda^{k+1}&=\bar \lambda^{k}+\sigma(\bar u^{k+1}-\bar z^{k+1}).
\end{aligned} \right.
\end{equation}

Usually, it is expensive and unnecessary to exactly compute the solution of each subproblem even if it is feasible. Thus it is natural to use some iterative methods such as Krylov-based methods to solve the subproblems which are equivalent to large scale or ill-conditioned linear systems. The inexact ADMM algorithm in finite dimension space and its convergence results under certain error criterion have been studied extensively recently (see \cite{Chen2017, Li2015}). 
Taking the inexactness of the solutions in the function space into account, Song et al. applied the inexact ADMM algorithm to Hilbert space for PDE-constrained optimization problems and presented the convergence results in \cite{Song2017Fe}. Given $(u^{0},z^{0},{\lambda}^{0})\in L^{2}(\Omega)\times {\rm dom}(\delta_{U_{ad}}(\cdot))\times L^{2}(\Omega)$, parameters $\sigma>0$, $\tau \in \left(0,\dfrac{1+\sqrt{5}}{2}\right)$. Let$\lbrace\epsilon_{k}\rbrace_{k=0}^{\infty}$ be a sequence satisfying $\lbrace\epsilon_{k}\rbrace_{k=0}^{\infty}\subseteq[0,+\infty)$ and $\sum_{k=0}^{\infty}\epsilon_{k}<\infty$, the iterative scheme of the inexact ADMM in function space for (\ref{eqn:reduced problem with linear constraint}) is as follows:
\begin{equation}
\left\{\begin{aligned}
\text{Step 1}:\ &\text{Compute an approximation solution $u^{k+1}$ of}\\
&\min_{u} \hat{J}(u)+\langle \lambda^{k},u-z^{k}\rangle_{L^{2}(\Omega)}+\dfrac{\sigma}{2}\|u-z^{k}\|_{L^{2}(\Omega)}^{2}\\
&\text{such that the error vector}\  \delta_{u}^{k+1}:=\nabla \hat{J}(u^{k+1})+\lambda^{k}+\sigma(u^{k+1}-z^{k}) \ \text{satisfies }  \|\delta_{u}^{k+1}\|_{L^{2}(\Omega)}\leq \epsilon_{k}.\\
\text{Step 2}: \ &z^{k+1}={\rm argmin} \delta_{U_{ad}}(z)+\langle \lambda^{k},u^{k+1}-z\rangle_{L^{2}(\Omega)}+\dfrac{\sigma}{2}\|u^{k+1}-z\|_{L^{2}(\Omega)}^{2}.\\
\text{Step 3}: \ &\lambda^{k+1}=\lambda^{k}+\tau \sigma(u^{k+1}-z^{k+1}).
\end{aligned} \right.
\end{equation}

In \cite{Song2017Fe}, the discretized problem is considered and the level of discretization is fixed. In this paper, instead of considering the discretized problem, the finite element method is employed to discretize the subproblems in each iteration of the inexact ADMM algorithm. This strategy gives the freedom and flexibility to discretize the subproblems by different discretization schemes. The total error of the proposed algorithm for (\ref{eqn:reduced problem with linear constraint}) is consisted of two parts: the discretization error and the iteration error resulted from inexactly solving the discretized subproblems. For these two errors, our algorithm can be considered as an approximation of exact ADMM in function space, thus we regard it as an inexact ADMM algorithm in function space. In order to guarantee the convergence behavior of our algorithm, we consider controlling the mesh size and the inexactness of solving the subproblems. 

In the classical finite element based ADMM-type algorithm to solve the PDE constrained optimization problem, the subproblems are always discretized on a fixed mesh size, which results in large scale optimization problems. Thus it is important to consider reducing the computation cost. Multi-grid method is a modern field of research starting with the works of Brandt \cite{Brandt1977Multi} and Hackbusch \cite{Hackbusch1985, Hackbusch1978On}. It is well known that multi-grid method solves elliptic problems with optimal computational complexity. Motivated by the idea of applying multi-grid method to tackle infinite dimension problems by Newton method in \cite{Deuflhard2011newton, Hackbusch1985}, we apply the multi-level strategy to our algorithm. It is important to point out that in the initial stage of the algorithm,  the iteration precision is required to be relatively low, which means using coarse mesh is sufficient. While as the iteration process proceeds, the iteration precision is supposed to be higher and higher. In this case, using finer mesh is necessary. It is obvious that the strategy of gradually refining the grid can strongly reduce the computation cost and make the algorithm faster than computing the problem on a fixed mesh size. 

The main contribution of this paper is that we give the reasonable strategies of gradually refining the grid and inexactly solving the subproblems, and propose an efficient convergent multi-level ADMM (mADMM) algorithm.  Specifically, we apply the ADMM algorithm in function space, then employ the standard piecewise linear finite element approach and implement the strategy of gradually refining the grid to the related subproblems appearing in each iteration. For the discretized subproblems, we use inexact strategies to solve them. For example, to solve the smooth subproblems, we can use some iterative methods such as Krylov-based methods. Thanks to the above strategies, we can solve the problem more efficiently and reduce the computation cost significantly. Moreover, we give the convergence analysis as well as the iteration complexity results $o(1/k)$ for the mADMM algorithm.

The paper is organized as follows. In section 2, we apply the inexact ADMM algorithm in function space first, then use finite element method to discretize the associated subproblems and propose the multi-level ADMM (mADMM) algorithm. Section 3 gives the convergence analysis and the iteration complexity of the proposed mADMM algorithm. The numerical results are given in section 4. Section 5 contains a brief summary of this paper.

\section{A multi-level ADMM algorithm}
\label{sec:2}
In this section, we apply the \emph{`optimize-discretize-optimize'} strategy and propose an efficient convergent multi-level ADMM (mADMM) algorithm in Algorithm \ref{alg:multi-level ADMM algorithm}. The strategies of gradually refining the grid and inexactly solving the subproblems are given to guarantee the convergence and efficiency of the mADMM algorithm. 
\subsection{The mADMM algorithm}
\label{sec:2.1}
Based on the inexact ADMM in function space, to numerically solve problem (\ref{eqn:reduced problem with linear constraint}), we consider the finite element method. Specifically, piecewise linear functions are utilized to discretize the variables in the related subproblems appearing in each iteration of the inexact ADMM in function space.
We first consider a family of regular and quasi-uniform triangulations 
$\lbrace \mathcal{T}_{h}\rbrace$ of $\bar \Omega$, i.e. $\bar \Omega=\bigcup_{T\in \mathcal{T}_{h}}\bar T$. With each element $T\in  \mathcal{T}_{h}$, we define the diameter of the set $T$ by $\rho_{T}:={\rm diam} \ T$ 
and let $\sigma_{T}$ denotes the diameter of the largest ball contained in $T$. The mesh size of the grid is defined by $h:={\rm max}_{T\in \mathcal{T}_{h}}\ \rho_{T}$. We suppose the following standard assumption in the context of error estimates holds (see\cite{Hinze2010Variational,Hinze2009Optimization}).

\begin{assumption}(Regular and quasi-uniform triangulations)
There exist two positive constants $\kappa$ and $\tau$ such that 
\begin{equation*}
\frac{\rho_{T}}{\sigma_{T}}\leq \kappa, \quad \frac{h}{\rho_{T}}\leq \tau \end{equation*}
hold for all $T\in \mathcal{T}_{h}$ and all $h>0$. Moreover, let us define 
$\bar \Omega_{h}=\bigcup_{T\in \mathcal{T}_{h}}\bar T$ and let $\Omega_{h}\subseteq \Omega$ and $\Gamma_{h}$ denote its interior and its boundary, respectively. In the case that $\Omega$ is a convex polyhedral domain, we have $\Omega=\Omega_{h}$. In the case that $\Omega$ is a domain with a $C^{1,1}$- boundary $\Gamma$, we assume that $\bar \Omega_{h}$ is convex and that all boundary vertices of $\bar \Omega_{h}$ are contained in $\Gamma$, such that 
\begin{equation*}
|\Omega\backslash\Omega_{h}|\leq ch^{2},
\end{equation*}
where $|\cdot|$ denotes the measure of the set, and $c>0$ is a constant.
\end{assumption}

 For the state variable $y$ and the control variable $u$, we choose the same discretized state space and discretized control space defined by 
\begin{align}
&Y_{h}:=\{y_{h}\in C( \bar{\Omega})|y_{{h}|T}\in \mathcal{P}_{1}, \forall{T\in \mathcal{T}_{h},y_{h}=0 \ {\rm in} \ \bar{\Omega} \setminus\Omega_{h}}\},\\
&U_{h}:=\{u_{h}\in C( \bar{\Omega})|u_{{h}|T}\in \mathcal{P}_{1}, \forall{T\in \mathcal{T}_{h},u_{h}=0 \ {\rm in} \ \bar{\Omega} \setminus\Omega_{h}}\},
\end{align}
where $\mathcal{P}_{1}$ denotes the space of polynomials of degree less than or equal to 1. 
For given source term $y_{r}$ and $u\in L^{2}(\Omega)$, let the discretized state associated with $u$ denoted by $y_{h}(u)$, which is defined as the unique solution for the discretized weak formulation of the state equation (\ref{eqn:pde}):
\begin{equation}\label{eqn:weak formulation}
\int_{\Omega_{h}}\left(\sum_{i,j=1}^{n}a_{ij}(y_{h})_{x_{i}}(v_{h})_{x_{j}}+c_{0}y_{h}v_{h}\right)dx=\int_{\Omega_{h}}(u+y_{r})v_{h}dx, \ \ \forall v_{h}\in Y.
\end{equation}
Moreover, $y_{h}$ can be expressed by $y_{h}(u)=S_{h}(u+y_{r})$, where $S_{h}$ denotes the discretized version of the solution operator $S$. Then we have the following well-known conclusion about error estimates.
\begin{lemma}
\label{lem:error estimates}{\rm (\cite{Ciarlet2002The}, Thm. 4.4.6)}
For a given $u\in L^{2}(\Omega)$, let $y$ be the unique weak solution of  the state equation (\ref{eqn:pde}) and $y_{h}$ be the unique solution of (\ref{eqn:weak formulation}). Then there exists a constant $c>0$ independent of $h$, $u$ and $y_{r}$ such that
\begin{equation*}
\|y-y_{h}\|_{L^{2}(\Omega)}+h\|\nabla y-\nabla y_{h}\|_{L^{2}(\Omega)}\leq ch^{2}(\|u\|_{L^{2}(\Omega)}+\|y_{r}\|_{L^{2}(\Omega)}).
\end{equation*}
In particular, this implies $\|S-S_{h}\|_{L^{2}(\Omega)\rightarrow L^{2}(\Omega)}\leq ch^{2}$ and $\|S-S_{h}\|_{L^{2}(\Omega)\rightarrow H^{1}(\Omega)}\leq ch$.
\end{lemma}

For the given triangulation $\mathcal{T}_{h}$ with nodes $\{x_{i}\}_{i=1}^{N_{h}}$, let $\{\phi_{i}(x)\}_{i=1}^{N_{h}}$ be a set of nodal basis functions associated with nodes $\{x_{i}\}_{i=1}^{N_{h}}$ which spans $Y_{h}$, $U_{h}$ and satisfies the properties: 
\begin{equation}
\phi_{i}(x)\geqslant0,\ \ \|\phi_{i}(x)\|_{\infty}=1\ \ \ \forall{i=1,...,N_{h}}, \ \ \sum\limits_{i=1}^{N_{h}}\phi_{i}(x)=1.
\end{equation}
Then $y_{h}\in Y_{h}$, $u_{h}\in U_{h}$ can be represented as $y_{h}=\sum\limits_{i=1}^{N_{h}}y_{i}\phi_{i},$
$u_{h}=\sum\limits_{i=1}^{N_{h}}u_{i}\phi_{i}$ and we have $y_{h}(x_{i})=y_{i}$, $u_{h}(x_{i})=u_{i}$.
The other variables and operators in the subproblems of the inexact ADMM in function space are all discretized by piecewise linear functions similarly. 

Before we give the proposed algorithm to solve the problem (\ref{eqn:reduced problem with linear constraint}), we first introduce the definition of node interpolation operator $I_{h}$ for the following convergence analysis. 
\begin{definition}
For a given regular and quasi-uniform triangulation $\mathcal{T}_{h}$ of $\Omega$ with nodes $\{x_{i}\}_{i=1}^{N_{h}}$, let $\{\phi_{i}(x)\}_{i=1}^{N_{h}}$ denotes a set of associated nodal basis functions. Then the interpolation operator is defined as 
\begin{equation*}
I_{h}w(x):=\sum_{i=1}^{N_{h}}w(x_{i})\phi_{i}(x)  \ \  {\rm for \ any} \ w \in L^{2}(\Omega).
\end{equation*}
\end{definition}
Moreover, about the interpolation error estimate, we have the following result.
\begin{lemma}
\label{lem:interpolation error estimate}{\rm (\cite{Ciarlet2002The}, Thm. 3.1.6)}
For all $w \in W^{k+1,p}(\Omega), k\geq 0, p, q\in[0,+\infty)$ and $0\leq m \leq k+1$, we have
\begin{equation*}
\|w-I_{h}w\|_{W^{m,q}(\Omega)}\leq c_{I}h^{k+1-m}\|w\|_{W^{k+1,p}(\Omega)},
\end{equation*}
where $c_{I}$ is a constant which is independent of the mesh size $h$.
\end{lemma}
In the classical finite element based algorithm, the discretization mesh size is fixed in advance. When computing on the finer mesh, the scale of the discretized problem becomes larger and the computation cost becomes larger. So it is important to consider reducing the computation cost. 
In the first several iterations of the algorithm, the iteration precision is not that good, which means using coarse mesh will not make the precision worse, but reduce the computation amount. While as the iteration process proceeds, the iteration precision becomes higher and higher. In this case, using finer mesh is necessary. Thus we introduce the idea of gradually refining the grid.
Specifically, in the initial iteration we choose coarse mesh and obtain a solution first, then use the interpolation operator to project the obtained solution to the finer mesh. Finally, we apply appropriate numerical methods to solve the subproblems in the finer mesh and obtain a more precise solution, so on and so forth. 

Let the mesh size of the $k$th iterate denotes by $h_{k}, k\in \mathbb{Z}, k\geq 1$, then the discretized reduced cost function is defined as follows,
\begin{equation}
{\hat J_{h_{k+1}}}(u_{h_{k+1}}):=\dfrac{1}{2}\|S_{h_{k+1}}(u_{h_{k+1}}+I_{h_{k+1}}y_{r})-I_{h_{k+1}}y_{d}\|^{2}_{L^{2}(\Omega)}+\dfrac{\alpha}{2}\|u_{h_{k+1}}\|^{2}_{L^{2}(\Omega)}.
\end{equation}
Let $U_{ad, h_{k+1}}$ denotes the discretized feasible set,
\begin{equation}
U_{a d, h_{k+1}} :=U_{h_{k+1}} \cap U_{a d}=\left\{z_{h_{k+1}}=\sum_{i=1}^{N_{h_{k+1}}} z_{i} \phi_{i}(x) | a \leq z_{i} \leq b, \forall i=1, \cdots, N_{h_{k+1}}\right\} \subset U_{a d}.
\end{equation}
Moreover, we define  
$\lambda_{h_{k+1}}^{k}:=I_{h_{k+1}}\lambda_{h_{k}}^{k}$ and $z_{h_{k+1}}^{k}:=I_{h_{k+1}}z_{h_{k}}^{k}$, 
where $I_{h_{k+1}}$ denotes the node interpolation operator. 
Based on the above representations, we present the iterative scheme of the multi-level ADMM algorithm (mADMM) in Algorithm \ref{alg:multi-level ADMM algorithm}.

\begin{algorithm}[H]
\caption{multi-level ADMM (mADMM) algorithm for (\ref{eqn:reduced problem with linear constraint})
}
\label{alg:multi-level ADMM algorithm}
\begin{algorithmic}
\STATE{\textbf{Input:}} $(u^{0},z^{0},{\lambda}^{0})\in L^{2}(\Omega)\times {\rm dom}(\delta_{U_{ad}}(\cdot))\times L^{2}(\Omega)$, parameters $\sigma>0$, $\tau \in \left(0,\dfrac{1+\sqrt{5}}{2}\right)$. Let$\lbrace\xi_{k+1}\rbrace_{k=0}^{\infty}$ be a sequence satisfying $\lbrace\xi_{k+1}\rbrace_{k=0}^{\infty}\subseteq[0,+\infty)$ and $\sum_{k=0}^{\infty}\xi_{k+1}<\infty$, mesh sizes $\lbrace h_{k} \rbrace_{k=0}^{\infty}$ of each iteration satisfy $\sum_{k=0}^{\infty}h_{k+1}< \infty.$
Set $k=0$.
\STATE{\textbf{Output:}} $u^{k}_{h_{k}},z^{k}_{h_{k}},{\lambda}^{k}_{h_{k}}$.
\STATE{\textbf{Step 1}} Compute an approximation solution of
\begin{flalign}
\min_{u_{h_{k+1}}} 
{\hat J_{h_{k+1}}}(u_{h_{k+1}})+\langle{\lambda}_{h_{k+1}}^{k},u_{h_{k+1}}-{z}_{h_{k+1}}^{k}\rangle_{L^{2}(\Omega)}
+\dfrac{\sigma}{2}\|u_{h_{k+1}}-{z}_{h_{k+1}}^{k}\|^{2}_{L^{2}(\Omega)} \nonumber
\end{flalign}
\qquad \ \ \ \ \
such that the error vector 
$\delta_{u,h_{k+1}}^{k+1}:=\nabla {\hat J_{h_{k+1}}}(u_{h_{k+1}})+{\lambda}_{h_{k+1}}^{k}+\sigma(u_{h_{k+1}}^{k+1}-{z}_{h_{k+1}}^{k})$ satisfies \\
\qquad \ \ \ \ \ $\|\delta_{u,h_{k+1}}^{k+1}\|_{L^{2}(\Omega)}\leq \xi_{k+1}.$
\STATE{\textbf{Step 2}} Compute $z^{k+1}_{h_{k+1}}$ as follows:
$$z^{k+1}_{h_{k+1}}={\rm argmin} \delta_{U_{ad,h_{k+1}}}(z_{h_{k+1}})+\langle {\lambda}_{h_{k+1}}^{k},u^{k+1}_{h_{k+1}}-z_{h_{k+1}}\rangle_{L^{2}(\Omega)}+\dfrac{\sigma}{2}\|u^{k+1}_{h_{k+1}}-z_{h_{k+1}}\|^{2}_{L^{2}(\Omega)}.$$
\STATE{\textbf{Step 3}} Compute $$\lambda^{k+1}_{h_{k+1}}=I_{h_{k+1}}\lambda^{k}_{h_{k}}+\tau \sigma(u^{k+1}_{h_{k+1}}-z^{k+1}_{h_{k+1}}).$$
\STATE{\textbf{Step 4}} If a termination criterion is met, stop; else, set $k:=k+1$ and go to Step 1.
\end{algorithmic}
\end{algorithm}

\begin{remark}
In order to guarantee the sequence $\lbrace\xi_{k+1}\rbrace_{k=0}^{\infty}\subseteq[0,+\infty)$ satisfies $\sum_{k=0}^{\infty}\xi_{k+1}<\infty$. In the numerical experiment, we can choose $\xi_{k+1}=\frac{1}{(k+1)^{2}}$ as an example.
\end{remark}
\begin{remark}
In order to guarantee the mesh sizes $\lbrace h_{k} \rbrace_{k=0}^{\infty}$ of each iteration satisfy $\sum_{k=0}^{\infty}h_{k+1}< \infty.$ In the numerical experiment, we choose $h_{k}=2^{-(k+4)}$ in Example \ref{ex1}, $h_{k}=\sqrt{2}/2^{(k+3)}$ in Example \ref{ex2}.
\end{remark}
\subsection{Numerical computation of the subproblems in Algorithm \ref{alg:multi-level ADMM algorithm}}
For $u=\sum_{i=1}^{N_{h}}u_{i}\phi_{i}\in U_{h}$, $y=\sum_{i=1}^{N_{h}}y_{i}\phi_{i}\in Y_{h}$, let $\boldsymbol{\rm u}=(u_{1},...,u_{N_{h}})$, $\boldsymbol{\rm y}=(y_{1},...,y_{N_{h}})$ be the relative coefficient vectors respectively. Let the $L^{2}$-projections of $y_{r}$ and $y_{d}$ onto $Y_{h}$ be $y_{r,h}:=\sum_{i=1}^{N_{h}}y_{r}^{i}\phi_{i}(x)$ and $y_{d,h}:=\sum_{i=1}^{N_{h}}y_{d}^{i}\phi_{i}(x)$, respectively. Similarly, $\boldsymbol{\rm y}_{r}=(y_{r}^{1},y_{r}^{2}, ...,y_{r}^{N_{h}})$ and $\boldsymbol{\rm y}_{d}=(y_{d}^{1},y_{d}^{2}, ...,y_{d}^{N_{h}})$ denote their coefficient vectors.
The stiffness matrix and the mass matrix are defined as:
\begin{equation}
K_{h}:=(a(\phi_{i},\phi_{j}))_{i,j=1}^{N_{h}}, \ M_{h}:=\left( \int_{\Omega_{h}}\phi_{i}\phi_{j}dx \right)_{i,j=1}^{N_{h}}.
\end{equation}
Furthermore, we define 
\begin{align}
f(\boldsymbol{\rm u})&:=\dfrac{1}{2}\|K_{h}^{-1}M_{h}(\boldsymbol{\rm u}+\boldsymbol{\rm y}_{r})-\boldsymbol{\rm y}_{d}\|_{M_{h}}^{2}+\dfrac{\alpha}{2}\|\boldsymbol{\rm u}\|_{M_{h}}^{2}, \\
g(\boldsymbol{\rm z})&:=\delta_{[a,b]^{N_{h}}}(\boldsymbol{\rm z}).
\end{align}
Then we can obtain the matrix-vector form of Algorithm \ref{alg:multi-level ADMM algorithm} in Algorithm \ref{alg:matrix-vector form}.

\begin{algorithm}[H]
\caption{matrix-vector form of the mADMM algorithm}
\label{alg:matrix-vector form}
\begin{algorithmic}
\STATE{\textbf{Input:}} $(\boldsymbol{\rm u}^{0},\boldsymbol{\rm z}^{0},\boldsymbol{\rm \lambda}^{0})\in \mathbb{R}^{N_{h}}\times{[a,b]^{N_{h}}}\times\mathbb{R}^{N_{h}}$, parameters $\sigma>0$, $\tau \in \left(0,\dfrac{1+\sqrt{5}}{2}\right)$. Let$\lbrace\xi_{k+1}\rbrace_{k=0}^{\infty}$ be a sequence satisfying $\lbrace\xi_{k+1}\rbrace_{k=0}^{\infty}\subseteq[0,+\infty)$ and $\sum_{k=0}^{\infty}\xi_{k+1}<\infty$, mesh sizes $\lbrace h_{k} \rbrace_{k=0}^{\infty}$ of each iteration satisfy $O(\sum_{k=0}^{\infty}h_{k+1})\leq \infty.$ Set $k=0$.
\STATE{\textbf{Output:}} $\boldsymbol{\rm u}^{k},\boldsymbol{\rm z}^{k},\boldsymbol{\rm \lambda}^{k}$.
\STATE{\textbf{Step 1}} Compute an approximation solution of
$$\min_{\boldsymbol{\rm u}} f(\boldsymbol{\rm u})+\langle \boldsymbol{\rm I}_{h_{k+1}}\boldsymbol{\rm \lambda}_{h_{k}}^{k},M_{h_{k+1}}(\boldsymbol{\rm u}-\boldsymbol{\rm I}_{h_{k+1}}\boldsymbol{\rm z}^{k}_{h_{k}})\rangle+\dfrac{\sigma}{2}\|\boldsymbol{\rm u}-\boldsymbol{\rm I}_{h_{k+1}}\boldsymbol{\rm z}^{k}_{h_{k}}\|_{M_{h_{k+1}}}^{2},$$
\qquad \ \ \ \ \ such that the error vector $$\boldsymbol{\delta}_{{\boldsymbol u},h_{k+1}}^{k}:=\nabla f(\boldsymbol{\rm u}^{k+1}_{h_{k+1}})+M_{h_{k+1}}^{T}\boldsymbol{\rm I}_{h_{k+1}}\boldsymbol{\rm \lambda}^{k}_{h_{k}}+\sigma M_{h_{k+1}}(\boldsymbol{\rm u}^{k+1}_{h_{k+1}}-\boldsymbol{\rm I}_{h_{k+1}}\boldsymbol{\rm z}^{k}_{h_{k}})$$ 
\qquad \ \ \ \ \ satisfies  $\|\boldsymbol{\delta}_{{\boldsymbol u},h_{k+1}}^{k}\|_{2}\leq \xi_{k+1}$.
\STATE{\textbf{Step 2}} Compute $\boldsymbol{\rm z}^{k}$ as follows:
$$\boldsymbol{\rm z}^{k+1}_{h_{k+1}}={\rm argmin} \ g(\boldsymbol{\rm z})+\langle \boldsymbol{\rm I}_{h_{k+1}}\boldsymbol{\rm \lambda}^{k}_{h_{k}},M_{h_{k+1}}(\boldsymbol{\rm z}-\boldsymbol{\rm u}^{k+1}_{h_{k+1}})\rangle+\dfrac{\sigma}{2}\|\boldsymbol{\rm u}^{k+1}_{h_{k+1}}-\boldsymbol{\rm z}\|_{M_{h_{k+1}}}^{2}.$$
\STATE{\textbf{Step 3}} Compute $$\boldsymbol{\rm \lambda}^{k+1}_{h_{k+1}}=\boldsymbol{\rm I}_{h_{k+1}}\boldsymbol{\rm \lambda}^{k}_{h_{k}}+\tau\sigma(\boldsymbol{\rm u}^{k+1}_{h_{k+1}}-\boldsymbol{\rm z}^{k+1}_{h_{k+1}}).$$
\STATE{\textbf{Step 4}} If a termination criterion is met, stop; else, set $k:=k+1$ and go to Step 1.
\end{algorithmic}
\end{algorithm}
Let 
$\boldsymbol{\rm y}^{k+1}:=K_{h}^{-1}M_{h}(\boldsymbol{\rm u}^{k+1}+\boldsymbol{\rm y}_{r}),\
\boldsymbol{\rm p}^{k+1}:=K_{h}^{-1}M_{h}(\boldsymbol{\rm y}_{d}-\boldsymbol{\rm y}^{k+1})$
denotes the discretized state and the discretized adjoint state respectively. Then the $\boldsymbol {\rm u}$-subproblem at the $k$-th iterate is equivalent to solving the following linear system:
\begin{equation}
\left[                
  \begin{array}{ccc}  
    M_{h} & 0 & K_{h}\\  
    0 &  (\alpha+\sigma)M_{h} & -M_{h}\\ 
    K_{h} & -M_{h} & 0\\ 
  \end{array}
\right]
\begin{array}{ccc}  
\left[       
  \begin{array}{ccc}           
    \boldsymbol{\rm y}^{k+1} \\  
    \boldsymbol{\rm u}^{k+1}\\ 
    \boldsymbol{\rm p}^{k+1}\\ 
\end{array}
\right]
\end{array}
=
\left[ 
\begin{array}{ccc}           
    M_{h}\boldsymbol{\rm y}_{d} \\  
    M_{h}(-\boldsymbol{\rm I}_{h_{k+1}}\boldsymbol{\rm \lambda}^{k}+\sigma \boldsymbol{\rm I}_{h_{k+1}}{\boldsymbol{\rm z}}^{k})\\ 
    M_{h}\boldsymbol{\rm y}_{r}\\ 
\end{array}   
\right].    
\end{equation}
Since $\boldsymbol{\rm p}^{k+1}=(\alpha+\sigma)\boldsymbol{\rm u}^{k+1}+\boldsymbol{\rm I}_{h_{k+1}}\boldsymbol{\rm \lambda}^{k}-\sigma \boldsymbol{\rm I}_{h_{k+1}}\boldsymbol{\rm z}^{k}$, we eliminate the variable $\boldsymbol{\rm p}$, then the above problem can be rewritten in the following reduced form without any computational cost:
\begin{equation}
\label{equ:linear}
\left[                
  \begin{array}{ccc}  
    M_{h} &  K_{h}(\alpha+\sigma) \\  
    -K_{h} &  M_{h} \\ 
  \end{array}
\right]
\begin{array}{ccc}  
\left[       
  \begin{array}{ccc}           
    \boldsymbol{\rm y}^{k+1} \\  
    \boldsymbol{\rm u}^{k+1}\\ 
\end{array}
\right]
\end{array}
=
\left[ 
\begin{array}{ccc}           
  M_{h}\boldsymbol{\rm y_{d}}-K_{h}\boldsymbol{\rm I}_{h_{k+1}}(\boldsymbol{\rm \lambda}^{k}+\sigma \boldsymbol{\rm I}_{h_{k+1}}\boldsymbol{\rm z}^{k}))\\
  M_{h}\boldsymbol{\rm y}_{r}
\end{array}   
\right].   
\end{equation}
The equivalent linear system (\ref{equ:linear}) can be solved inexactly by some Krylov-based method.

Finally, we give a terminal condition of Algorithm \ref{alg:matrix-vector form}. Let $\epsilon$ be a given accuracy tolerance, we terminate the algorithm when $\eta<\epsilon$, where the accuracy of a numerical solution is measured by the 
residual
$\eta:={\rm max}\lbrace\eta_{1},\eta_{2},\eta_{3},\eta_{4},\eta_{5}\rbrace,$
where
\begin{equation}
\begin{aligned}
\eta_{1}&=\frac{\|K_{h}\boldsymbol{\rm y}-M_{h}\boldsymbol{\rm u}-M_{h}\boldsymbol{\rm y_{r}}\|}{1+\|M_{h}{\boldsymbol{\rm y}}_{r}\|}, \ 
\eta_{2}=\frac{\|M_{h}(\boldsymbol{\rm u}-\boldsymbol{\rm z})\|}{1+\|\boldsymbol{\rm u}\|},\
\eta_{3}=\frac{\|M_{h}(\boldsymbol{\rm y}-{\boldsymbol{\rm y}}_{d})+K_{h}\boldsymbol{\rm p}\|}{1+\|M_{h}\boldsymbol{\rm y}_{d}\|},\\
\eta_{4}&=\frac{\|\alpha M_{h}\boldsymbol{\rm u}-M_{h}\boldsymbol{\rm p}+M_{h}\boldsymbol{\rm \lambda}\|}{1+\|\boldsymbol{\rm u}\|},\
\eta_{5}=\frac{\|\boldsymbol{\rm z}-\Pi_{[a,b]^{N_{h}}}(\boldsymbol{\rm z}+M_{h}\boldsymbol{\rm \lambda})\|}{1+\|\boldsymbol{\rm z}\|}.
\end{aligned}
\end{equation}

\section{Convergence analysis}
\label{sec:3}
In this section, based on the convergence results of the inexact ADMM in function space in Theorem \ref{convergence functional space}, we give the convergence analysis and the iteration complexity of the proposed mADMM algorithm. 
\begin{theorem}{\rm(\cite{Song2017Fe}, Thm. 2.5)}
\label{convergence functional space}
Suppose that the operator $L$ is uniformly elliptic. 
Let $(y^{\ast},u^{\ast},z^{\ast},p^{\ast},\lambda^{\ast})$ be the KKT point of (P), the sequence $\lbrace(u^{k},z^{k},\lambda^{k})\rbrace$ is generated by the inexact ADMM algorithm for (\ref{eqn:reduced problem with linear constraint}) with the associated state $\lbrace y_{k}\rbrace$ and adjoint state $\lbrace p_{k}\rbrace$, then we have
\begin{align*}
&\lim \limits_{k\rightarrow \infty}\lbrace \|u^{k}-u^{\ast}\|_{L^{2}(\Omega)}+\|z^{k}-z^{\ast}\|_{L^{2}(\Omega)}+\|\lambda^{k}-\lambda^{\ast}\|_{L^{2}(\Omega)}\rbrace=0,\\
&\lim \limits_{k\rightarrow \infty}\lbrace \|y^{k}-y^{\ast}\|_{H_{0}^{1}(\Omega)}+\|p^{k}-p^{\ast}\|_{H_{0}^{1}(\Omega)}\rbrace=0.
\end{align*}
Moreover, there exists a constant $C$ only depends on the initial point $(u^{0}, z^{0}, \lambda^{0})$ and the optimal solution $(u^{\ast}, z^{\ast}, \lambda^{\ast})$ such that for $k\geq 1$,
\begin{equation*}
\min_{1\leq i\leq k}{R(u^{i}, z^{i}, \lambda^{i})}\leq \frac{C}{k}, \ \ \lim \limits_{k\rightarrow \infty}(k\times \min_{1\leq i\leq k}{R(u^{i}, z^{i}, \lambda^{i})})=0,
\end{equation*}
where the function $R: (u, z, \lambda)\rightarrow [0,\infty)$ defined as 
\begin{equation}
R(u, z, \lambda):= \|\nabla {\hat J}(u)+\lambda\|^{2}_{L^{2}(\Omega)}+{\rm dist}^{2}(0, -\lambda+\partial \delta_{U_{ad}}(z))+\|u-z\|^{2}_{L^{2}(\Omega)}.
\end{equation}
\end{theorem}

For the convenience of proving the convergence results, we give the iterative scheme of the multi-level discretized ADMM for (\ref{eqn:reduced problem with linear constraint}) and present a lemma to measure the gap between the solution sequences obtained by the ADMM in function space and the multi-level discretized ADMM. 

Given $(u^{0},z^{0},{\lambda}^{0})\in L^{2}(\Omega)\times {\rm dom}(\delta_{U_{ad}}(\cdot))\times L^{2}(\Omega)$, parameters $\sigma>0$, $\tau \in \left(0,\dfrac{1+\sqrt{5}}{2}\right)$. The mesh sizes $\lbrace h_{k} \rbrace_{k=0}^{\infty}$ of each iteration satisfy $\sum_{k=0}^{\infty}h_{k+1}< \infty.$ Then the iterative format of the multi-level discretized ADMM algorithm is as follows:
\begin{equation}
\left\{\begin{aligned}
\bar u^{k+1}_{h_{k+1}}&={\rm argmin}\hat J_{h_{k+1}}(u_{h_{k+1}})+\langle\bar{\lambda}_{h_{k+1}}^{k},u_{h_{k+1}}-\bar{z}_{h_{k+1}}^{k}\rangle_{L^{2}(\Omega)}+\dfrac{\sigma}{2}\|u_{h_{k+1}}-\bar{z}_{h_{k+1}}^{k}\|_{L^{2}(\Omega)}^{2},\\
\bar z^{k+1}_{h_{k+1}}&={\rm argmin} \delta_{U_{ad,h_{k+1}}}(z_{h_{k+1}})+\langle {\bar \lambda}_{h_{k+1}}^{k},\bar u^{k+1}_{h_{k+1}}-z_{h_{k+1}}\rangle_{L^{2}(\Omega)}+\dfrac{\sigma}{2}\|\bar u^{k+1}_{h_{k+1}}-z_{h_{k+1}}\|_{L^{2}(\Omega)}^{2},\\
\bar \lambda^{k+1}_{h_{k+1}}&=I_{h_{k+1}}\bar \lambda^{k}_{h_{k}}+\sigma(\bar u^{k+1}_{h_{k+1}}-\bar z^{k+1}_{h_{k+1}}).
\end{aligned} \right.
\end{equation}

\begin{lemma}\label{lem:continuous exact and discretized exact}
Let the initial point be $(u^{0},z^{0};\lambda^{0})\in L^{2}(\Omega)\times {\rm dom}(\delta_{U_{ad}}(\cdot))\times L^{2}(\Omega)$, then
 $\|\bar{u}^{k+1}-\bar{u}^{k+1}_{h_{k+1}}\|_{L^{2}(\Omega)}=\|\bar{z}^{k+1}-\bar{z}^{k+1}_{h_{k+1}}\|_{L^{2}(\Omega)}=\|\bar{\lambda}^{k}-\bar{\lambda}^{k}_{h_{k+1}}\|_{L^{2}(\Omega)}=O(h_{k+1})$, $\forall k\geqslant 1,$
and
\begin{equation*}
\sum_{k=0}^{\infty}\|\bar{u}^{k+1}-\bar{u}^{k+1}_{h_{k+1}}\|_{L^{2}(\Omega)}=O\left(\sum_{k=0}^{\infty}h_{k+1}\right).
\end{equation*}
\end{lemma}
\begin{proof}
We employ the mathematical induction to prove the conclusion.
While $k=1$, with the definition of $\lambda_{h_{k+1}}^{k}:=I_{h_{k+1}}\lambda_{h_{k}}^{k}$ and the interpolation error estimate in Lemma \ref{lem:interpolation error estimate}, we have
\begin{equation}
\|\bar{\lambda}^{0}-\bar{\lambda}^{0}_{h_{1}}\|_{L^{2}(\Omega)}=\|\lambda^{0}-I_{h_{1}}\lambda^{0}\|_{L^{2}(\Omega)}=O(h_{1}), 
\end{equation}
Then we can easily obtain that $\|\bar{u}^{1}-\bar{u}^{1}_{h_{1}}\|_{L^{2}(\Omega)}=O(h_{1})$. The proof is similar to the case $k>1$, here we omit it.

While $k>1$, we assume for
$\forall j\leq k,$ we have $\|\bar{u}^{j}-\bar{u}^{j}_{h_{j}}\|_{L^{2}(\Omega)}=\|\bar{\lambda}^{j-1}-\bar{\lambda}^{j-1}_{h_{j}}\|_{L^{2}(\Omega)}=O(h_{j})$.
Then for $z$-subproblems in exact ADMM in function space and the multi-level discretized ADMM, we know $\bar{z}^{k}$ and $\bar{z}^{k}_{h_{k}}$ satisfy the following optimality conditions:
\begin{equation}
\bar{z}^{k}=\Pi_{[a,b]}\left(\bar{u}^{k}+\frac{\bar{\lambda}^{k-1}}{\sigma}\right),\
\bar{z}^{k}_{h_{k}}=\Pi_{[a,b]}\left(\bar{u}^{k}_{h_{k}}+\frac{\bar{\lambda}^{k-1}_{h_{k}}}{\sigma}\right).
\end{equation}
Then subtracting the above two equalities we obtain
\begin{equation}
\begin{aligned}
\|\bar{z}^{k}-\bar{z}^{k}_{h_{k}}\|_{L^{2}(\Omega)}
&=\left\|\Pi_{[a,b]}\left(\bar{u}^{k}+\dfrac{\bar{\lambda}^{k-1}}{\sigma}\right)-\Pi_{[a,b]}\left(\bar{u}^{k}_{h_{k}}+\dfrac{\bar{\lambda}^{k-1}_{h_{k}}}{\sigma}\right)\right\|_{L^{2}(\Omega)}
\\
&\leq \left\|\bar{u}^{k}+\dfrac{\bar{\lambda}^{k-1}}{\sigma}-\bar{u}^{k}_{h_{k}}-\dfrac{\bar{\lambda}^{k-1}_{h_{k}}}{\sigma}\right\|_{L^{2}(\Omega)}\\
&\leq \|\bar{u}^{k}-\bar{u}^{k}_{h_{k}}\|_{L^{2}(\Omega)}+\dfrac{1}{\sigma}\|\bar{\lambda}^{k-1}-\bar{\lambda}^{k-1}_{h_{k}}\|_{L^{2}(\Omega)}\\
&=O(h_{k}).
\end{aligned}
\end{equation}
For the multiplier $\bar{\lambda}^{k}$ and $\bar{\lambda}^{k}_{h_{k}}$, 
\begin{equation}
\bar{\lambda}^{k}=\bar{\lambda}^{k-1}+\sigma(\bar{u}^{k}-\bar{z}^{k}), \
\bar{\lambda}^{k}_{h_{k}}=\bar{\lambda}^{k-1}_{h_{k}}+\sigma(\bar{u}^{k}_{h_{k}}-\bar{z}^{k}_{h_{k}}),
\end{equation}
we can get the estimate
\begin{equation}\label{equ:lambda}
\begin{aligned}
\|\bar{\lambda}^{k}-\bar{\lambda}^{k}_{h_{k}}\|_{L^{2}(\Omega)}
&=\|\bar{\lambda}^{k-1}-\bar{\lambda}^{k-1}_{h_{k}}+\sigma(\bar{u}^{k}-\bar{u}^{k}_{h_{k}})-\sigma(\bar{z}^{k}-\bar{z}^{k}_{h_{k}})\|_{L^{2}(\Omega)}\\
&\leq \|\bar{\lambda}^{k-1}-\bar{\lambda}^{k-1}_{h_{k}}\|_{L^{2}(\Omega)}+ \sigma\|\bar{u}^{k}-\bar{u}^{k}_{h_{k}}\|_{L^{2}(\Omega)}+\sigma\|\bar{z}^{k}-\bar{z}^{k}_{h_{k}}\|_{L^{2}(\Omega)}\\
&=O(h_{k}).
\end{aligned}
\end{equation}
For $u$-subproblems, $\bar{u}^{k+1}$ and $\bar{u}^{k+1}_{h_{k+1}}$ satisfy the following optimality conditions respectively,
\begin{align}
&S^{\ast}[S(\bar{u}^{k}+y_{r})-y_{d}]+\alpha\bar{u}^{k+1}+\bar{\lambda}^{k}+\sigma(\bar{u}^{k+1}-\bar{z}^{k})=0,\\
&S^{\ast}_{h_{k+1}}[S_{h_{k+1}}(\bar{u}_{h_{k+1}}^{k+1}+I_{h_{k+1}}y_{r})-I_{h_{k+1}}y_{d}]+\alpha\bar{u}^{k+1}_{h_{k+1}}+\bar{\lambda}^{k}_{h_{k+1}}+\sigma(\bar{u}^{k+1}_{h_{k+1}}-\bar{z}^{k}_{h_{k+1}})=0.
\end{align}
Then we know from the above two equalities that
\begin{equation}
\begin{aligned}
0=&
S^{\ast}S\bar{u}^{k+1}-S^{\ast}_{h_{k+1}}S_{h_{k+1}}\bar{u}^{k+1}_{h_{k+1}}+S^{\ast}Sy_{r}-S^{\ast}_{h_{k+1}}S_{h_{k+1}}I_{h_{k+1}}y_{r}\\&-S^{\ast}y_{d}+S^{\ast}_{h_{k+1}}I_{h_{k+1}}y_{d}+(\alpha+\sigma)(\bar{u}^{k+1}-\bar{u}^{k+1}_{h_{k+1}})+\bar{\lambda}^{k}-\bar{\lambda}^{k}_{h_{k+1}}-\sigma(\bar{z}^{k}-\bar{z}^{k}_{h_{k+1}}),
\end{aligned}
\end{equation}
so 
\begin{equation}
\begin{aligned}
\label{equality}
(\alpha+\sigma)(\bar{u}^{k+1}-\bar{u}^{k+1}_{h_{k+1}})=&
-(S^{\ast}S\bar{u}^{k+1}-S^{\ast}_{h_{k+1}}S_{h_{k+1}}\bar{u}^{k+1}_{h_{k+1}}+S^{\ast}Sy_{r}-S^{\ast}_{h_{k+1}}S_{h_{k+1}}I_{h_{k+1}}y_{r}
\\&-S^{\ast}y_{d}+S^{\ast}_{h_{k+1}}I_{h_{k+1}}y_{d}+\bar{\lambda}^{k}-\bar{\lambda}^{k}_{h_{k+1}}-\sigma(\bar{z}^{k}-\bar{z}^{k}_{h_{k+1}}))\\
=&-(E_{1}+E_{2}+E_{3}+E_{4}+E_{5}),
\end{aligned}
\end{equation}
where we define
\begin{equation*}
\begin{aligned}
E_{1}&:=S^{\ast}S\bar{u}^{k+1}-S^{\ast}_{h_{k+1}}S_{h_{k+1}}\bar{u}^{k+1}_{h_{k+1}},\\
E_{2}&:=S^{\ast}Sy_{r}-S^{\ast}_{h_{k+1}}S_{h_{k+1}}I_{h_{k+1}}y_{r},\\
E_{3}&:=-S^{\ast}y_{d}+S^{\ast}_{h_{k+1}}I_{h_{k+1}}y_{d},\\
E_{4}&:=\bar{\lambda}^{k}-\bar{\lambda}^{k}_{h_{k+1}},\\
E_{5}&:=-\sigma(\bar{z}^{k}-\bar{z}^{k}_{h_{k+1}}).
\end{aligned}
\end{equation*}

For the term $E_{1}$, we make use of the decomposition, \begin{equation}
\begin{aligned}
\|E_{1}\|_{L^{2}(\Omega)}&=\|S^{\ast}S\bar{u}^{k+1}-S^{\ast}S_{h_{k+1}}\bar{u}^{k+1}+S^{\ast}S_{h_{k+1}}\bar{u}^{k+1}-S^{\ast}_{h_{k+1}}S_{h_{k+1}}\bar{u}^{k+1}+S^{\ast}_{h_{k+1}}S_{h_{k+1}}\bar{u}^{k+1}-S^{\ast}_{h_{k+1}}S_{h_{k+1}}\bar{u}^{k+1}_{h_{k+1}}\|_{L^{2}(\Omega)}\\
&\leq\|S^{\ast}(S-S_{h_{k+1}})\bar{u}^{k+1}\|_{L^{2}(\Omega)}+\|(S^{\ast}-S^{\ast}_{h_{k+1}})S_{h_{k+1}}\bar{u}^{k+1}\|_{L^{2}(\Omega)}+\|S^{\ast}_{h_{k+1}}S_{h_{k+1}}(\bar{u}^{k+1}-\bar{u}^{k+1}_{h_{k+1}})\|_{L^{2}(\Omega)}.
\end{aligned}
\end{equation}
From the well known error estimate $\|S-S_{h}\|_{L^{2}(\Omega)\rightarrow L^{2}(\Omega)}=O(h^{2})$ in lemma \ref{lem:error estimates} and the property that $S^{\ast},S_{h_{k+1}}$ are bounded linear operators, we have
\begin{align}
&\|S^{\ast}(S-S_{h_{k+1}})\bar{u}^{k+1}\|_{L^{2}(\Omega)}\leq
\|S^{\ast}\|_{L^{2}(\Omega)\rightarrow L^{2}(\Omega)}\|S-S_{h_{k+1}}\|_{L^{2}(\Omega)\rightarrow L^{2}(\Omega)}\|\bar{u}^{k+1}\|_{L^{2}(\Omega)}=O(h_{k+1}^{2}),\\
&\|(S^{\ast}-S^{\ast}_{h_{k+1}})S_{h_{k+1}}\bar{u}^{k+1}\|_{L^{2}(\Omega)}\leq\|S^{\ast}-S^{\ast}_{h_{k+1}}\|_{L^{2}(\Omega)\rightarrow L^{2}(\Omega)}\|S_{h_{k+1}}\bar{u}^{k+1}\|_{L^{2}(\Omega)}=O(h_{k+1}^{2}).
\end{align}
Hence, there exists a constant $\hat C$ such that
\begin{equation}
\|E_{1}\|_{L^{2}(\Omega)}\leq \hat Ch_{k+1}^{2}+\|S^{\ast}_{h_{k+1}}S_{h_{k+1}}\|_{L^{2}(\Omega)\rightarrow L^{2}(\Omega)}\|\bar{u}^{k+1}-\bar{u}^{k+1}_{h_{k+1}}\|_{L^{2}(\Omega)}.
\end{equation}

Similarly, based on the property of the projection operator $\|y_{r}-I_{h_{k+1}}y_{r}\|_{L^{2}(\Omega)}=\|y_{d}-I_{h_{k+1}}y_{d}\|_{L^{2}(\Omega)}=O(h_{k+1})$ in Lemma \ref{lem:interpolation error estimate}, for the term $E_{2}$, we have
\begin{equation}
\begin{aligned}
\|E_{2}\|_{L^{2}(\Omega)}=&\|S^{\ast}Sy_{r}-S^{\ast}_{h_{k+1}}Sy_{r}+S^{\ast}_{h_{k+1}}Sy_{r}-S^{\ast}_{h_{k+1}}S_{h_{k+1}}y_{r}+S^{\ast}_{h_{k+1}}S_{h_{k+1}}y_{r}-S^{\ast}_{h_{k+1}}S_{h_{k+1}}I_{h_{k+1}}y_{r}\|_{L^{2}(\Omega)}\\
\leq& \|(S^{\ast}-S^{\ast}_{h_{k+1}})Sy_{r}\|_{L^{2}(\Omega)}+\|S^{\ast}_{h_{k+1}}(S-S_{h_{k+1}})y_{r}\|_{L^{2}(\Omega)}+\|S^{\ast}_{h_{k+1}}S_{h_{k+1}}(y_{r}-I_{h_{k+1}}y_{r})\|_{L^{2}(\Omega)}\\
\leq & \|S^{\ast}-S^{\ast}_{h_{k+1}}\|_{L^{2}(\Omega)\rightarrow L^{2}(\Omega)}\|Sy_{r}\|_{L^{2}(\Omega)}+\|S^{\ast}_{h_{k+1}}\|_{L^{2}(\Omega)\rightarrow L^{2}(\Omega)}\|S-S_{h_{k+1}}\|_{L^{2}(\Omega)\rightarrow L^{2}(\Omega)}\|y_{r}\|_{L^{2}(\Omega)}\\
&+\|S^{\ast}_{h_{k+1}}S_{h_{k+1}}\|_{L^{2}(\Omega)\rightarrow L^{2}(\Omega)}\|y_{r}-I_{h_{k+1}}y_{r}\|_{L^{2}(\Omega)}\\
=&O(h_{k+1}),
\end{aligned}
\end{equation}
and
\begin{equation}
\begin{aligned}
\|E_{3}\|_{L^{2}(\Omega)}&=\|-S^{\ast}y_{d}+S^{\ast}I_{h_{k+1}}y_{d}-S^{\ast}I_{h_{k+1}}y_{d}+S^{\ast}_{h_{k+1}}I_{h_{k+1}}y_{d}\|_{L^{2}(\Omega)}\\
&\leq \|S^{\ast}(y_{d}-I_{h_{k+1}}y_{d})\|_{L^{2}(\Omega)}+\|(S^{\ast}-S^{\ast}_{h_{k+1}})I_{h_{k+1}}y_{d}\|_{L^{2}(\Omega)}\\
&=O(h_{k+1}).
\end{aligned}
\end{equation}

For the term $E_{4}$,
\begin{equation}
\begin{aligned}
\|E_{4}\|_{L^{2}(\Omega)}
&=\|\bar{\lambda}^{k}-\bar{\lambda}^{k}_{h_{k}}+\bar{\lambda}^{k}_{h_{k}}-I_{h_{k+1}}\bar{\lambda}^{k}_{h_{k}}\|_{L^{2}(\Omega)}\\
&\leq \|\bar{\lambda}^{k}-\bar{\lambda}^{k}_{h_{k}}\|_{L^{2}(\Omega)}+\|\bar{\lambda}^{k}_{h_{k}}-I_{h_{k+1}}\bar{\lambda}^{k}_{h_{k}}\|_{L^{2}(\Omega)}\\
&=O(h_{k}+h_{k+1}), 
\end{aligned}
\end{equation}
where we used (\ref{equ:lambda}), the property of the projection operator $\|\bar{\lambda}^{k}_{h_{k}}-I_{h_{k+1}}\bar{\lambda}^{k}_{h_{k}}\|_{L^{2}(\Omega)}=O(h_{k+1})$ and the property of the mesh size $h_{k}>h_{k+1}$. Moreover, as the mesh sizes satisfy $\sum_{k=0}^{\infty}h_{k+1}< \infty,$ there exists a constant $C_{k+1}$ such that $h_{k}<C_{k+1}h_{k+1}$, then we have 
\begin{equation}
\|E_{4}\|_{L^{2}(\Omega)}=\|\bar{\lambda}^{k}-\bar{\lambda}^{k}_{h_{k+1}}\|_{L^{2}(\Omega)}=O(h_{k+1}).
\end{equation}

Similarly,
\begin{equation}
\begin{aligned}
\|E_{5}\|_{L^{2}(\Omega)}
&=\sigma\|\bar{z}^{k}-\bar{z}^{k}_{h_{k}}+\bar{z}^{k}_{h_{k}}-\bar{z}^{k}_{h_{k+1}}\|_{L^{2}(\Omega)}\\
&\leq \sigma\|\bar{z}^{k}-\bar{z}^{k}_{h_{k}}\|_{L^{2}(\Omega)}+\sigma\|\bar{z}^{k}_{h_{k}}-\bar{z}^{k}_{h_{k+1}}\|_{L^{2}(\Omega)}\\
&=O(h_{k+1}).
\end{aligned}
\end{equation}

Then with the fact that operators $S^{\ast}_{h_{k+1}}, S_{h_{k+1}}$ are bounded linear operators, we know from the equality (\ref{equality}) and the estimations of $L^{2}$ norms of $\lbrace E_{i}\rbrace_{i=1}^{5}$ above that 
\begin{equation}
\|\bar{u}^{k+1}-\bar{u}^{k+1}_{h_{k+1}}\|_{L^{2}(\Omega)}=O(h_{k+1}).
\end{equation}
Moreover, we have
\begin{equation}
\begin{aligned}
\|\bar{z}^{k+1}-\bar{z}^{k+1}_{h_{k+1}}\|_{L^{2}(\Omega)}
&=\left\|\Pi_{[a,b]}\left(\bar{u}^{k+1}+\dfrac{\bar{\lambda}^{k}}{\sigma}\right)-\Pi_{[a,b]}\left(\bar{u}^{k+1}_{h_{k+1}}+\dfrac{\bar{\lambda}^{k}_{h_{k+1}}}{\sigma}\right)\right\|_{L^{2}(\Omega)}
\\
&\leq \|\bar{u}^{k+1}-\bar{u}^{k+1}_{h_{k+1}}\|_{L^{2}(\Omega)}+\dfrac{1}{\sigma}\|(\bar{\lambda}^{k}-\bar{\lambda}^{k}_{h_{k+1}})\|_{L^{2}(\Omega)}\\
&=O(h_{k+1}).
\end{aligned}
\end{equation}
Hence the conclusion holds for the case $k+1$ and we can get the assertion
\begin{equation}
\sum_{k=0}^{\infty}\|\bar{u}^{k+1}-\bar{u}^{k+1}_{h_{k+1}}\|_{L^{2}(\Omega)}=O\left(\sum_{k=0}^{\infty}h_{k+1}\right).
\end{equation}
\end{proof}

Similar to the Lemma \ref{lem:continuous exact and discretized exact}, we have the following lemma.
\begin{lemma}\label{lem:continuous inexact and discretized inexact}
Let the initial point be $(u^{0},z^{0};\lambda^{0})\in L^{2}(\Omega)\times {\rm dom}(\delta_{U_{ad}}(\cdot))\times L^{2}(\Omega)$, then
 $\|u^{k+1}-{u}^{k+1}_{h_{k+1}}\|_{L^{2}(\Omega)}=\|z^{k+1}-{z}^{k+1}_{h_{k+1}}\|_{L^{2}(\Omega)}=\|{\lambda}^{k}-{\lambda}^{k}_{h_{k+1}}\|_{L^{2}(\Omega)}=O(h_{k+1}+\delta_{u,h_{k+1}}^{k+1})$, $\forall k\geqslant 1,$
and
\begin{equation*}
\sum_{k=0}^{\infty}\|{u}^{k+1}-{u}^{k+1}_{h_{k+1}}\|_{L^{2}(\Omega)}=O\left(\sum_{k=0}^{\infty}\left(h_{k+1}+\delta_{u,h_{k+1}}^{k+1}\right)\right).
\end{equation*}
\end{lemma}
\begin{proof}
We employ the mathematical induction to prove the conclusion. The proof is similar to Lemma \ref{lem:continuous exact and discretized exact}, here we do not talk about it in detail.
\end{proof}
To prove the convergence of the mADMM algorithm, 
let $(\tilde{u}_{h_{k+1}}^{k+1},\tilde{z}^{k+1}_{h_{k+1}})$ represents the exact solutions of the $(k+1)$th iteration of Algorithm \ref{alg:multi-level ADMM algorithm}:
\begin{align}
\label{tildeu}
&\tilde{u}_{h_{k+1}}^{k+1}:={\rm argmin}{\hat J}(u_{h_{k+1}})+\langle\lambda_{h_{k+1}}^{k},u_{h_{k+1}}-z_{h_{k+1}}^{k}\rangle+\dfrac{\sigma}{2}\|u_{h_{k+1}}-z_{h_{k+1}}^{k}\|_{L^{2}(\Omega)}^{2},\\
\label{tildez}
&\tilde{z}^{k+1}_{h_{k+1}}:={\rm argmin}\delta_{U_{ad,h}}(z_{h_{k+1}})+\langle{\lambda}^{k}_{h_{k+1}},\tilde{u}^{k+1}_{h_{k+1}}-z_{h_{k+1}}\rangle+\dfrac{\sigma}{2}\|\tilde{u}^{k+1}_{h_{k+1}}-z_{h_{k+1}}\|_{L^{2}(\Omega)}^{2}.
\end{align}
The following lemma gives the gap between $(\tilde{u}_{h_{k+1}}^{k+1}, \tilde{z}_{h_{k+1}}^{k+1})$ and $({u}_{h_{k+1}}^{k+1}, {z}_{h_{k+1}}^{k+1})$. 
\begin{lemma}\rm(\cite{Song2017Fe}, Lemma 4.4)
\label{lmm:tildeu}
For any $k\geqslant 0$, we have
\begin{align*}
&\|\tilde{u}_{h_{k+1}}^{k+1}-u_{h_{k+1}}^{k+1}\|_{L^{2}(\Omega)}\leq\rho\|\delta_{u,h_{k+1}}^{k+1}\|_{L^{2}(\Omega)},\\
&\|\tilde{z}_{h_{k+1}}^{k+1}-z_{h_{k+1}}^{k+1}\|_{L^{2}(\Omega)}\leq\rho\|\delta_{u,h_{k+1}}^{k+1}\|_{L^{2}(\Omega)},
\end{align*}
where $\rho:=\|[S^{\ast}_{h_{k+1}}S_{h_{k+1}}+(\alpha+\sigma)I]^{-1}\|_{{L^{2}(\Omega)}\rightarrow {L^{2}(\Omega)}}$. 
\end{lemma}

For the convenience of analyzing the non-ergodic iteration complexity, let $(u^{\ast}_{h_{k+1}}, z^{\ast}_{h_{k+1}}, \lambda^{\ast}_{h_{k+1}})$ denotes the KKT point of the discretized reduced problem with the mesh size $h_{k+1}$ 
\begin{equation}\label{discretized reduced problem}
\begin{aligned}
\min \limits_{u_{h_{k+1}},z_{h_{k+1}}}
&\frac{1}{2}\|S_{h_{k+1}}(u_{h_{k+1}}+I_{h_{k+1}}y_{r})-I_{h_{k+1}}y_d\|_{L^{2}(\Omega_{h})}^{2}
+\frac{\alpha}{2}\|u_{h_{k+1}}\|_{{L^{2}(\Omega_{h})}}^{2}+\delta_{U_{ad,h_{k+1}}}(z_{h_{k+1}})\\
{\rm s.t.}  \ \   \ \ \ & u_{h_{k+1}}=z_{h_{k+1}}.
\end{aligned} \tag{DRP}
\end{equation} 
Moreover, we provide a lemma and two propositions which are essential for analyzing the iteration complexity of our mADMM. 

\begin{lemma}\rm(\cite{Chen2017}, Lemma 6.1)
\label{lmm:complexity1}
If a sequence $\left\{a_{i}\right\} \in \mathbb{R}$ satisfies the following conditions: 
\begin{equation}
a_{i} \geq 0 \  for \ any \ i \geq 0 \ \ and \ \ \sum_{i=0}^{\infty} a_{i}=\bar a<\infty .
\end{equation}
Then we have $\min_{i=1,2, \cdots, k}\left\{a_{i}\right\} \leq \frac{\overline{a}}{k}$ and $\lim_{k \rightarrow \infty}\left\{k \cdot \min _{i=1,2, \cdots, k}\left\{a_{i}\right\}\right\}=0$.
\end{lemma}

\begin{proposition}
\label{lmm:complexity2}
Let $\left\{\left(u^{k}_{h_{k}}, z^{k}_{h_{k}}, \lambda^{k}_{h_{k}}\right)\right\}$ be the sequence generated by Algorithm \ref{alg:multi-level ADMM algorithm} and $\left\{\left(u^{\ast}_{h_{k}}, z^{\ast}_{h_{k}}, \lambda^{\ast}_{h_{k}}\right)\right\}$ denotes the KKT point of the discretized reduced problem. Then for $k\geq 0$ we have
\begin{align*} 
&\left\langle\delta^{k+1}_{u,h_{k+1}}, u^{k+1}_{h_{k+1}}-u^{\ast}_{h_{k+1}}\right\rangle+\frac{1}{2\tau\sigma}\|\lambda^{\ast}_{h_{k+1}}-\lambda^{k}_{h_{k+1}}\|_{L^{2}(\Omega)}^{2}+\frac{\sigma}{2}\|z^{\ast}_{h_{k+1}}-z^{k}_{h_{k+1}}\|_{L^{2}(\Omega)}^{2}\\
&-\frac{1}{2\tau\sigma}\|\lambda^{\ast}_{h_{k+1}}-\lambda^{k+1}_{h_{k+1}}\|_{L^{2}(\Omega)}^{2}-\frac{\sigma}{2}\|z^{\ast}_{h_{k+1}}-z^{k+1}_{h_{k+1}}\|_{L^{2}(\Omega)}^{2}\\
\geq &\alpha\|u^{k+1}_{h_{k+1}}-u^{\ast}_{h_{k+1}}\|_{L^{2}(\Omega)}^{2}+\frac{(3-\tau)\sigma}{2}\|u^{k+1}_{h_{k+1}}-z^{k+1}_{h_{k+1}}\|_{L^{2}(\Omega)}^{2}+\frac{\sigma}{2}\|u^{k+1}_{h_{k+1}}-z^{k}_{h_{k+1}}\|_{L^{2}(\Omega)}^{2}.
\end{align*}
\end{proposition}
\begin{proof}
First, for any $f_{1}, f_{1}^{\prime}, f_{2}, f_{2}^{\prime}\in L^{2}(\Omega)$, we have the following two important equalities hold
\begin{align}  
\label{equ1}
&\langle f_{1}, f_{2}\rangle_{L^{2}(\Omega)}=\frac{1}{2}\left(\|f_{1}\|_{L^{2}(\Omega)}^{2}+\|f_{2}\|_{L^{2}(\Omega)}^{2}-\|f_{1}-f_{2}\|_{L^{2}(\Omega)}^{2}\right)
=\frac{1}{2}\left(\|f_{1}+f_{2}\|_{L^{2}(\Omega)}^{2}-\|f_{1}\|_{L^{2}(\Omega)}^{2}-\|f_{2}\|_{L^{2}(\Omega)}^{2}\right),\\
\label{equ2}
&\left\langle f_{1}-f_{1}^{\prime}, f_{2}-f_{2}^{\prime}\right\rangle_{L^{2}(\Omega)}=\frac{1}{2}\left(\|f_{1}+f_{2}\|_{L^{2}(\Omega)}^{2}+\left\|f_{1}^{\prime}+f_{2}^{\prime}\right\|_{L^{2}(\Omega)}^{2}-\left\|f_{1}+f_{2}^{\prime}\right\|_{L^{2}(\Omega)}^{2}-\left\|f_{1}^{\prime}+f_{2}\right\|_{L^{2}(\Omega)}^{2}\right).
\end{align} 
The proof of the above two equalities can be easily obtained by the definition of $L^{2}-$norm. 

By the optimality conditions of the $u$-subproblem and $z$-subproblem corresponding to $u_{h_{k+1}}^{k+1}$ and $z_{h_{k+1}}^{k+1}$, we have 
\begin{align}
\label{optimality1}
&\nabla {\hat J_{h_{k+1}}}(u_{h_{k+1}}^{k+1})=S^{\ast}_{h_{k+1}}[S_{h_{k+1}}(u_{h_{k+1}}^{k+1}+I_{h_{k+1}}y_{r})-I_{h_{k+1}}y_{d}]+\alpha u^{k+1}_{h_{k+1}}=\delta_{u,h_{k+1}}^{k+1}-(\lambda^{k}_{h_{k+1}}+\sigma(u^{k+1}_{h_{k+1}}-z^{k}_{h_{k+1}})),\\
&\lambda^{k}_{h_{k+1}}+\sigma(u^{k+1}_{h_{k+1}}-z^{k+1}_{h_{k+1}}))\in \partial
\delta_{U_{ad,h_{k+1}}}(z_{h_{k+1}}^{k+1}).
\end{align}
Moreover, $\left\{\left(u^{\ast}_{h_{k}}, z^{\ast}_{h_{k}}, \lambda^{\ast}_{h_{k}}\right)\right\}$ denotes the KKT point of the discretized reduced problem, so it satisfies the following KKT system 
\begin{align}
\label{optimality2}
&\nabla {\hat J_{h_{k+1}}}(u_{h_{k+1}}^{\ast})=S^{\ast}_{h_{k+1}}[S_{h_{k+1}}(u_{h_{k+1}}^{\ast}+I_{h_{k+1}}y_{r})-I_{h_{k+1}}y_{d}]+\alpha u^{\ast}_{h_{k+1}}=-\lambda^{\ast}_{h_{k+1}},\\
&\lambda^{\ast}_{h_{k+1}}\in \partial
\delta_{U_{ad,h_{k+1}}}(z_{h_{k+1}}^{\ast}),\\
&u^{\ast}_{h_{k+1}}=z^{\ast}_{h_{k+1}}.
\end{align}
Then by combining (\ref{optimality1}) and (\ref{optimality2}), we obtain
\begin{equation}
\begin{aligned}
&\left\langle \nabla {\hat J_{h_{k+1}}}(u_{h_{k+1}}^{k+1})-\nabla {\hat J_{h_{k+1}}}(u_{h_{k+1}}^{\ast}), u_{h_{k+1}}^{k+1}-u_{h_{k+1}}^{\ast} \right\rangle\\
=&\left\langle S^{\ast}_{h_{k+1}}S_{h_{k+1}}(u_{h_{k+1}}^{k+1}-u_{h_{k+1}}^{\ast})+\alpha(u_{h_{k+1}}^{k+1}-u_{h_{k+1}}^{\ast}), u_{h_{k+1}}^{k+1}-u_{h_{k+1}}^{\ast}  \right\rangle\\
=&\|S_{h_{k+1}}(u_{h_{k+1}}^{k+1}-u_{h_{k+1}}^{\ast})\|_{L^{2}(\Omega)}^{2}+\alpha\|u_{h_{k+1}}^{k+1}-u_{h_{k+1}}^{\ast}\|_{L^{2}(\Omega)}^{2}.
\end{aligned}
\end{equation}
Moreover, the subdifferential operator $\partial
\delta_{U_{ad,h_{k+1}}}(z)$ is a maximal monotone operator, so the following inequality holds,
\begin{equation}
\left\langle \partial
\delta_{U_{ad,h_{k+1}}}(z_{h_{k+1}}^{k+1})-\partial
\delta_{U_{ad,h_{k+1}}}(z_{h_{k+1}}^{\ast}), z_{h_{k+1}}^{k+1}-z_{h_{k+1}}^{\ast} \right\rangle\geq 0.
\end{equation}
For the convenience of analyzing, we define $r_{h_{k+1}}^{k+1}=u_{h_{k+1}}^{k+1}-z_{h_{k+1}}^{k+1}$.
Therefore, we can derive that 
\begin{align}
&\left\langle \delta_{u,h_{k+1}}^{k+1}-(\tilde \lambda_{h_{k+1}}^{k}+\sigma(z_{h_{k+1}}^{k+1}-z_{h_{k+1}}^{k}))+\lambda_{h_{k+1}}^{\ast},  u_{h_{k+1}}^{k+1}-u_{h_{k+1}}^{\ast}\right\rangle\geq \alpha\|u_{h_{k+1}}^{k+1}-u_{h_{k+1}}^{\ast}\|_{L^{2}(\Omega)}^{2},\\
&\left\langle \lambda_{h_{k+1}}^{k}+\sigma r_{h_{k+1}}^{k+1}, z_{h_{k+1}}^{k+1}-z_{h_{k+1}}^{\ast} \right\rangle\geq 0.
\end{align}
Then adding the above two equalities we obtain
\begin{equation}
\label{estimate}
\left\langle \delta_{u,h_{k+1}}^{k+1},  u_{h_{k+1}}^{k+1}-u_{h_{k+1}}^{\ast}\right\rangle-\left\langle \lambda_{h_{k+1}}^{k}+\sigma r_{h_{k+1}}^{k+1}-\lambda_{h_{k+1}}^{\ast}, r_{h_{k+1}}^{k+1}\right\rangle-\sigma\left\langle z_{h_{k+1}}^{k+1}-z_{h_{k+1}}^{k}, u_{h_{k+1}}^{k+1}-u_{h_{k+1}}^{\ast} \right\rangle \geq \alpha\|u_{h_{k+1}}^{k+1}-u_{h_{k+1}}^{\ast}\|_{L^{2}(\Omega)}^{2}.
\end{equation}
Next, we estimate the last two terms on the left side separately,
\begin{equation}
\begin{aligned}
\label{estimate1}
&\left\langle \lambda_{h_{k+1}}^{\ast}-(\lambda_{h_{k+1}}^{k}+\sigma r_{h_{k+1}}^{k+1}), r_{h_{k+1}}^{k+1}\right\rangle
\\
=&\frac{1}{\tau\sigma}\left\langle \lambda_{h_{k+1}}^{\ast}-\lambda_{h_{k+1}}^{k}, \lambda_{h_{k+1}}^{k+1}-\lambda_{h_{k+1}}^{k} \right\rangle-\sigma\|r_{h_{k+1}}^{k+1}\|_{L^{2}(\Omega)}^{2}\\
=&\frac{1}{2\tau\sigma}(\|\lambda_{h_{k+1}}^{\ast}-\lambda_{h_{k+1}}^{k}\|_{L^{2}(\Omega)}^{2}+\|\lambda_{h_{k+1}}^{k+1}-\lambda_{h_{k+1}}^{k}\|_{L^{2}(\Omega)}^{2}-\|\lambda_{h_{k+1}}^{\ast}-\lambda_{h_{k+1}}^{k+1}\|_{L^{2}(\Omega)}^{2})-\sigma\|r_{h_{k+1}}^{k+1}\|_{L^{2}(\Omega)}^{2}\\
=&\frac{1}{2\tau\sigma}(\|\lambda_{h_{k+1}}^{\ast}-\lambda_{h_{k+1}}^{k}\|_{L^{2}(\Omega)}^{2}-\|\lambda_{h_{k+1}}^{\ast}-\lambda_{h_{k+1}}^{k+1}\|_{L^{2}(\Omega)}^{2})+\frac{(\tau-2)\sigma}{2}\|r_{h_{k+1}}^{k+1}\|_{L^{2}(\Omega)}^{2},
\end{aligned}
\end{equation}
where we used the equality (\ref{equ1}).

Moreover, by employing the equality (\ref{equ2}) and using $u^{\ast}_{h_{k+1}}=z^{\ast}_{h_{k+1}}$, we have
\begin{equation}
\begin{aligned}
\label{estimate2}
&\sigma\left\langle z_{h_{k+1}}^{k+1}-z_{h_{k+1}}^{k}, u_{h_{k+1}}^{\ast}-u_{h_{k+1}}^{k+1} \right\rangle \\
=&\sigma\left\langle z_{h_{k+1}}^{k+1}-z_{h_{k+1}}^{k}, -u_{h_{k+1}}^{k+1}-(-z_{h_{k+1}}^{\ast}) \right\rangle \\
=&\frac{\sigma}{2}(\|r_{h_{k+1}}^{k+1}\|_{L^{2}(\Omega)}^{2}+\|z_{h_{k+1}}^{k}-z_{h_{k+1}}^{\ast}\|_{L^{2}(\Omega)}^{2}-\|z_{h_{k+1}}^{k+1}-z_{h_{k+1}}^{\ast}\|_{L^{2}(\Omega)}^{2}
-\|z_{h_{k+1}}^{k}-u_{h_{k+1}}^{k+1}\|_{L^{2}(\Omega)}^{2}).
\end{aligned}
\end{equation}
Then, substituting (\ref{estimate1}), (\ref{estimate2}) into (\ref{estimate}), we can get the assertion of Proposition \ref{lmm:complexity2}. 
\end{proof}

\begin{proposition}
\label{lmm:complexity3}
Let $\left\{\left(u^{k}_{h_{k}}, z^{k}_{h_{k}}, \lambda^{k}_{h_{k}}\right)\right\}$ be the sequence generated by Algorithm \ref{alg:multi-level ADMM algorithm}, $\left\{\left(u^{\ast}_{h_{k}}, z^{\ast}_{h_{k}}, \lambda^{\ast}_{h_{k}}\right)\right\}$ denotes the KKT point of the discretized reduced problem and $\tilde{u}_{h_{k+1}}^{k+1}$, $\tilde{z}^{k+1}_{h_{k+1}}$ 
defined in (\ref{tildeu}), (\ref{tildez}), respectively. Then for $k\geq 0$ we have
\begin{align*} 
&\frac{1}{2\tau\sigma}\|\lambda^{\ast}_{h_{k+1}}-\lambda^{k}_{h_{k+1}}\|_{L^{2}(\Omega)}^{2}+\frac{\sigma}{2}\|z^{\ast}_{h_{k+1}}-z^{k}_{h_{k+1}}\|_{L^{2}(\Omega)}^{2}-\frac{1}{2\tau\sigma}\|\lambda^{\ast}_{h_{k+1}}-\tilde \lambda^{k+1}_{h_{k+1}}\|_{L^{2}(\Omega)}^{2}-\frac{\sigma}{2}\|z^{\ast}_{h_{k+1}}-\tilde z^{k+1}_{h_{k+1}}\|_{L^{2}(\Omega)}^{2}\\
\geq & \alpha\|\tilde u^{k+1}_{h_{k+1}}-u^{\ast}_{h_{k+1}}\|_{L^{2}(\Omega)}^{2}+\frac{(3-\tau)\sigma}{2}\|\tilde u^{k+1}_{h_{k+1}}-\tilde z^{k+1}_{h_{k+1}}\|_{L^{2}(\Omega)}^{2}+\frac{\sigma}{2}\|\tilde u^{k+1}_{h_{k+1}}-z^{k}_{h_{k+1}}\|_{L^{2}(\Omega)}^{2}.
\end{align*}
\end{proposition}

\begin{proof}
For the proof of Proposition \ref{lmm:complexity3}, by substituting $\tilde u^{k+1}_{h_{k+1}}$ and $\tilde z^{k+1}_{h_{k+1}}$ for $u^{k+1}_{h_{k+1}}$ and $z^{k+1}_{h_{k+1}}$ in the proof of Proposition \ref{lmm:complexity2}, we can get the assertion.
\end{proof}
Finally, based on the above results, the convergence results of Algorithm \ref{alg:multi-level ADMM algorithm} is given by the following theorem.
\begin{theorem}\label{convergence theorem}
Suppose that the operator $L$ is uniformly elliptic. 
Let $(y^{\ast},u^{\ast},z^{\ast},p^{\ast},\lambda^{\ast})$ be the KKT point of {\rm(\ref{eqn:orginal problems})}, 
$(u^{k}_{h_{k}},z^{k}_{h_{k}},\lambda^{k}_{h_{k}})$ is obtained in the $k$th iterate of Algorithm \ref{alg:multi-level ADMM algorithm}, where we suppose the mesh sizes $\lbrace h_{k} \rbrace_{k=0}^{\infty}$ of each iteration satisfy $\sum_{k=0}^{\infty}h_{k+1}<\infty,$ and the error vector $\delta_{u,h_{k+1}}^{k+1}$ satisfies $\|\delta_{u,h_{k+1}}^{k+1}\|_{L^{2}(\Omega)}\leq \xi_{k+1}$, $\sum_{k=0}^{\infty}\xi_{k+1}<\infty.$
Then we have
\begin{align*}
&\lim \limits_{k\rightarrow \infty}\lbrace \|u^{k}_{h_{k}}-u^{\ast}\|_{L^{2}(\Omega)}+\|z^{k}_{h_{k}}-z^{\ast}\|_{L^{2}(\Omega}+\|\lambda^{k}_{h_{k}}-\lambda^{\ast}\|_{L^{2}(\Omega)}\rbrace=0,\\
&\lim \limits_{k\rightarrow \infty}\lbrace \|y^{k}_{h_{k}}-y^{\ast}\|_{H_{0}^{1}(\Omega)}+\|p^{k}_{h_{k}}-p^{\ast}\|_{H_{0}^{1}(\Omega)}\rbrace=0.
\end{align*}
Moreover, there exists a constant $\tilde C$ only depends on the initial point $(u^{0}, z^{0}, \lambda^{0})$ and the optimal solution $(u^{\ast}, z^{\ast}, \lambda^{\ast})$ such that for $k\geq 1$,
\begin{equation*}
\min_{1\leq i\leq k}{R_{h_{i}}}(u^{i}_{h_{i}}, z^{i}_{h_{i}}, \lambda^{i}_{h_{i}})\leq \frac{\tilde C}{k}, \ \ \lim \limits_{k\rightarrow \infty}(k\times \min_{1\leq i\leq k}{R_{h_{i}}}(u^{i}_{h_{i}}, z^{i}_{h_{i}}, \lambda^{i}_{h_{i}}))=0,
\end{equation*}
where $R_{h_{i}}: (u^{i}_{h_{i}}, z^{i}_{h_{i}}, \lambda^{i}_{h_{i}})\rightarrow [0,\infty)$ is defined as 
\begin{equation*}
R_{h_{i}}(u^{i}_{h_{i}}, z^{i}_{h_{i}}, \lambda^{i}_{h_{i}}):=\|\nabla {\hat J_{h_{i}}}(u^{i}_{h_{i}})+\lambda_{h_{i}}^{i-1}\|^{2}_{L^{2}(\Omega)}+{\rm dist}^{2}(0,-\lambda_{h_{i}}^{i-1}+\partial \delta_{U_{ad,h_{i}}}(z^{i}_{h_{i}}))+\|u^{i}_{h_{i}}-z^{i}_{h_{i}}\|^{2}_{L^{2}(\Omega)}.
\end{equation*}
\end{theorem}
\begin{proof}
By the optimality condition of the $u$-subproblem in the ADMM in function space, we have
\begin{eqnarray} 
S^{\ast}[S(\bar{u}^{k+1}+y_{r})-y_{d}]+\alpha\bar{u}^{k+1}+\bar \lambda^{k}+\sigma(\bar{u}^{k+1}-\bar z^{k})=0.
\end{eqnarray} 
As we know, the error between the inexact solution and exact solution contains two parts, error from gradually refining the grid and error from the inexactly solving the subproblems. We take them into consideration together as a total error, let $u_{h_{k+1}}^{k+1}$ represents the inexact solution of the $(k+1)$th iteration, then from the optimality condition of the $u$-subproblem, we have
\begin{equation}
S^{\ast}[S(u_{h_{k+1}}^{k+1}+y_{r})-y_{d}]+\alpha u_{h_{k+1}}^{k+1}+\lambda^{k}+\sigma(u_{h_{k+1}}^{k+1}-z^{k})=\delta_{u}^{k+1}.
\end{equation}
Moreover, by the optimality conditions of the $u-$subproblem corresponding to $u_{h_{k+1}}^{k+1}$ and $\bar u_{h_{k+1}}^{k+1}$ in Algorithm \ref{alg:multi-level ADMM algorithm} and multi-level discretized ADMM, we have
\begin{align}
&S^{\ast}_{h_{k+1}}[S_{h_{k+1}}(u_{h_{k+1}}^{k+1}+I_{h_{k+1}}y_{r})-I_{h_{k+1}}y_{d}]+\alpha u_{h_{k+1}}^{k+1}+\lambda^{k}_{h_{k+1}}+\sigma(u_{h_{k+1}}^{k+1}-z^{k}_{h_{k+1}})=\delta_{u,h_{k+1}}^{k+1},\\
&S^{\ast}_{h_{k+1}}[S_{h_{k+1}}(\bar u_{h_{k+1}}^{k+1}+I_{h_{k+1}}y_{r})-I_{h_{k+1}}y_{d}]+\alpha \bar u_{h_{k+1}}^{k+1}+\bar \lambda^{k}_{h_{k+1}}+\sigma(\bar u_{h_{k+1}}^{k+1}-\bar z^{k}_{h_{k+1}})=0.
\end{align}

Then we know from the four equalities above that 
\begin{equation}
\begin{aligned}
\delta_{u}^{k+1}=&\delta_{u}^{k+1}-\delta_{u,h_{k+1}}^{k+1}+\delta_{u,h_{k+1}}^{k+1}\\
=&\delta_{u,h_{k+1}}^{k+1}+S^{\ast}S(u_{h_{k+1}}^{k+1}-\bar u^{k+1})+(\alpha+\sigma)(u_{h_{k+1}}^{k+1}-\bar u^{k+1})+(\lambda^{k}-\bar \lambda^{k})-\sigma(z^{k}-\bar z^{k})\\
&+S^{\ast}_{h_{k+1}}S_{h_{k+1}}(\bar u_{h_{k+1}}^{k+1}-u_{h_{k+1}}^{k+1})+(\alpha+\sigma)(\bar u_{h_{k+1}}^{k+1}-u_{h_{k+1}}^{k+1})+(\bar \lambda^{k}_{h_{k+1}}-\lambda^{k}_{h_{k+1}})-\sigma(\bar z^{k}_{h_{k+1}}-z^{k}_{h_{k+1}})\\
=&\delta_{u,h_{k+1}}^{k+1}+(\alpha+\sigma)(\bar u_{h_{k+1}}^{k+1}-\bar u^{k+1})+
(\lambda^{k}-\lambda^{k}_{h_{k+1}})-\sigma(z^{k}-z^{k}_{h_{k+1}})+(\bar \lambda^{k}_{h_{k+1}}-\bar \lambda^{k})-\sigma(\bar z^{k}_{h_{k+1}}-\bar z^{k})\\
&+(S^{\ast}S-S^{\ast}_{h_{k+1}}S_{h_{k+1}})(u_{h_{k+1}}^{k+1}-\bar u_{h_{k+1}}^{k+1})+S^{\ast}S(\bar u_{h_{k+1}}^{k+1}-\bar u^{k+1}).
\end{aligned}
\end{equation} 
Moreover, we have the estimate 
\begin{equation}
\begin{aligned}
\label{equ:estimate}
&\|(S^{\ast}S-S^{\ast}_{h_{k+1}}S_{h_{k+1}})(u_{h_{k+1}}^{k+1}-\bar u_{h_{k+1}}^{k+1})\|_{L^{2}(\Omega)}\\
=&\|(S^{\ast}S-S^{\ast}S_{h_{k+1}})(u_{h_{k+1}}^{k+1}-\bar u_{h_{k+1}}^{k+1})\|_{L^{2}(\Omega)}
+\|(S^{\ast}S_{h_{k+1}}-S^{\ast}_{h_{k+1}}S_{h_{k+1}})(u_{h_{k+1}}^{k+1}-\bar u_{h_{k+1}}^{k+1})\|_{L^{2}(\Omega)}\\
\leq &\|S^{\ast}\|\|S-S_{h_{k+1}}\|\|u_{h_{k+1}}^{k+1}-\bar u_{h_{k+1}}^{k+1}\|_{L^{2}(\Omega)}+\|S^{\ast}-S^{\ast}_{h_{k+1}}\|\|S_{h_{k+1}}\|\|u_{h_{k+1}}^{k+1}-\bar u_{h_{k+1}}^{k+1}\|_{L^{2}(\Omega)}\\
= &O(h_{k+1}^{2}).
\end{aligned}
\end{equation} 

Then we know from (\ref{equ:estimate}), Lemma \ref{lem:continuous exact and discretized exact} and Lemma \ref{lem:continuous inexact and discretized inexact} that there exists constant C such that
\begin{equation}\label{equ: delta}
\begin{aligned}
\sum_{k=0}^{\infty}\|\delta_{u}^{k+1}\|_{L^{2}(\Omega)}
\leq \sum_{k=0}^{\infty}\|\delta_{u,h_{k+1}}^{k+1}\|_{L^{2}(\Omega)}+C\sum_{k=0}^{\infty}\left(h_{k+1}+\|\delta_{u,h_{k+1}}^{k+1}\|_{L^{2}(\Omega)}\right).
\end{aligned}
\end{equation} 
We know from Algorithm \ref{alg:multi-level ADMM algorithm} that the mesh sizes $\lbrace h_{k} \rbrace_{k=0}^{\infty}$ of each iteration satisfy 
$\sum_{k=0}^{\infty}h_{k+1}<\infty.$ The error vector of the multi-level ADMM satisfy
$\sum_{k=0}^{\infty}\|\delta_{u,h_{k+1}}^{k+1}\|_{L^{2}(\Omega)}\leq\sum_{k=0}^{\infty}\xi_{k+1}<\infty,$
where $\xi_{k+1}$ is the upper bound of $\delta_{u,h_{k+1}}^{k+1}$, i.e. $\|\delta_{u,h_{k+1}}^{k+1}\|_{L^{2}(\Omega)}\leq \xi_{k+1}$. Thus we have
\begin{equation}
\sum_{k=0}^{\infty}\|\delta_{u}^{k+1}\|_{L^{2}(\Omega)}
< \infty.
\end{equation} 
For the discretization error and the iteration error, Algorithm \ref{alg:multi-level ADMM algorithm} can be considered as an inexact ADMM algorithm in function space, then we know from Theorem \ref{convergence functional space} that the convergence of Algorithm \ref{alg:multi-level ADMM algorithm} is guaranteed. 

At last, we establish the proof of the iteration complexity results for the sequence generated by the mADMM.
First, by the optimality condition for $(u^{k+1}_{h_{k+1}},z^{k+1}_{h_{k+1}})$, we have
\begin{equation}
\begin{aligned}
&\delta_{u,h_{k+1}}^{k+1}-[\lambda^{k}_{h_{k+1}}+\sigma(u^{k+1}_{h_{k+1}}-z^{k}_{h_{k+1}})]=\nabla \hat J_{h_{k+1}}(u^{k+1}_{h_{k+1}}),\\
&\lambda^{k}_{h_{k+1}}+\sigma(u^{k+1}_{h_{k+1}}-z^{k+1}_{h_{k+1}})\in \partial \delta_{U_{ad},h_{k+1}}(z^{k+1}_{h_{k+1}}).
\end{aligned}
\end{equation}
Then by the definition of $R_{h_{k}}$, we derive
\begin{equation}\label{R inequality}
\begin{aligned}
R_{h_{k+1}}(u^{k+1}_{h_{k+1}},z^{k+1}_{h_{k+1}},\lambda^{k+1}_{h_{k+1}})=&\|\nabla \hat J_{h_{k+1}}(u^{k+1}_{h_{k+1}})+\lambda^{k}_{h_{k+1}}\|_{L^{2}(\Omega)}^{2}+{\rm dist}^{2}(0,-\lambda^{k}_{h_{k+1}}+\partial \delta_{U_{ad},h_{k+1}}(z^{k+1}_{h_{k+1}}))
\\&
+\|u^{k+1}_{h_{k+1}}-z^{k+1}_{h_{k+1}}\|_{L^{2}(\Omega)}^{2}\\
\leq &\|\delta_{u,h_{k+1}}^{k+1}-\sigma(u^{k+1}_{h_{k+1}}-z^{k}_{h_{k+1}})\|_{L^{2}(\Omega)}^{2}+(1+\sigma^{2})\|u^{k+1}_{h_{k+1}}-z^{k+1}_{h_{k+1}}\|_{L^{2}(\Omega)}^{2}\\
\leq & 2\|\delta_{u,h_{k+1}}^{k+1}\|_{L^{2}(\Omega)}^{2}+2\sigma^{2}\|u^{k+1}_{h_{k+1}}-z^{k}_{h_{k+1}}\|_{L^{2}(\Omega)}^{2}+(1+\sigma^{2})\|u^{k+1}_{h_{k+1}}-z^{k+1}_{h_{k+1}}\|_{L^{2}(\Omega)}^{2}.
\end{aligned}
\end{equation}

Next, for the convenience of giving an upper bound of $R_{h_{k+1}}(u^{k+1}_{h_{k+1}},z^{k+1}_{h_{k+1}},\lambda^{k+1}_{h_{k+1}})$, we define the following sequence $\theta_{k}$, $\bar \theta_{k}$ and $\tilde \theta_{k}$ with:
\begin{align}
&\theta^{k}=\left(\frac{1}{\sqrt{2 \tau \sigma}} \left( \lambda^{k}_{h_{k+1}}-\lambda^{\ast}_{h_{k+1}}\right), \sqrt{\frac{\sigma}{2}} \left(z^{k}_{h_{k+1}}-z^{\ast}_{h_{k+1}}\right)\right),\\
&\bar \theta^{k}=\left(\frac{1}{\sqrt{2 \tau \sigma}} \left( \lambda^{k}_{h_{k}}-\lambda^{\ast}_{h_{k}}\right), \sqrt{\frac{\sigma}{2}} \left(z^{k}_{h_{k}}-z^{\ast}_{h_{k}}\right)\right),\\
&\tilde \theta^{k}=\left(\frac{1}{\sqrt{2 \tau \sigma}} \left(\tilde \lambda^{k}_{h_{k}}-\lambda^{\ast}_{h_{k}}\right), \sqrt{\frac{\sigma}{2}} \left(\tilde z^{k}_{h_{k}}-z^{\ast}_{h_{k}}\right)\right).
\end{align}
First, we give an upper bound of $\theta^{k}$. We know from Lemma \ref{lmm:complexity3} that 
$\|\tilde \theta^{k+1}\|_{L^{2}(\Omega)}\leq \|\theta^{k}\|_{L^{2}(\Omega)},$ so 
\begin{equation}
\begin{aligned}
\|\theta^{k+1}\|_{L^{2}(\Omega)}&\leq \|\tilde \theta^{k+1}\|_{L^{2}(\Omega)}+\|\tilde \theta^{k+1}-\theta^{k+1}\|_{L^{2}(\Omega)},\\
&\leq \|\theta^{k}\|_{L^{2}(\Omega)}+\|\tilde \theta^{k+1}-\theta^{k+1}\|_{L^{2}(\Omega)}.
\end{aligned}
\end{equation}
We know from the definition of $\tilde \theta^{k+1}$, $\theta^{k+1}$ and Lemma \ref{lmm:tildeu} that
\begin{equation}
\label{theta}
\begin{aligned}
\|\tilde \theta^{k+1}-\theta^{k+1}\|_{L^{2}(\Omega)}=&\frac{1}{2 \tau \sigma}\left\|\tilde{\lambda}^{k+1}_{h_{k+1}}-\lambda^{k+1}_{h_{k+2}}\right\|_{L^{2}(\Omega)}+\frac{\sigma}{2}\left\|\tilde{z}^{k+1}_{h_{k+1}}-z^{k+1}_{h_{k+2}}\right\|_{L^{2}(\Omega)}\\
\leq& \frac{1}{2 \tau \sigma}\left\|\tilde{\lambda}^{k+1}_{h_{k+1}}-\lambda^{k+1}_{h_{k+1}}\right\|_{L^{2}(\Omega)}+\frac{1}{2 \tau \sigma}\left\|\lambda^{k+1}_{h_{k+1}}-I_{h_{k+2}}\lambda^{k+1}_{h_{k+1}}\right\|_{L^{2}(\Omega)}\\
&+\frac{\sigma}{2}\left\|\tilde{z}^{k+1}_{h_{k+1}}-z^{k+1}_{h_{k+1}}\right\|_{L^{2}(\Omega)}+\frac{\sigma}{2}\left\|z^{k+1}_{h_{k+1}}-I_{h_{k+2}}z^{k+1}_{h_{k+1}}\right\|_{L^{2}(\Omega)}\\
\leq& \frac{1}{2 \tau \sigma}\|\tilde{u}^{k+1}_{h_{k+1}}-{u}^{k+1}_{h_{k+1}}\|_{L^{2}(\Omega)}+\frac{\sigma}{2}\left\|\tilde{z}^{k+1}_{h_{k+1}}-z^{k+1}_{h_{k+1}}\right\|_{L^{2}(\Omega)}+{\tilde C_{1}}h_{k+2}\\
\leq& (\frac{1}{\tau \sigma}+\frac{\sigma}{2})\rho\|\delta_{u,h_{k+1}}^{k+1}\|_{L^{2}(\Omega)}+{\tilde C_{1}}h_{k+2},
\end{aligned}
\end{equation}
where $\tilde C_{1}$ is a constant.
So there exists a constant $C'_{1}$ such that for every $k$
\begin{equation}
\begin{aligned}
\label{es1}
\|\theta^{k+1}\|_{L^{2}(\Omega)}&\leq \|\theta^{k}\|_{L^{2}(\Omega)}+\|\tilde \theta^{k+1}-\theta^{k+1}\|_{L^{2}(\Omega)}\\
&\leq \|\theta^{k}\|_{L^{2}(\Omega)}+\left(\frac{1}{\tau \sigma}+\frac{\sigma}{2}\right)\rho\|\delta_{u,h_{k+1}}^{k+1}\|_{L^{2}(\Omega)}+{\tilde C_{1}}h_{k+2},\\
&\leq  \|\theta^{k-1}\|_{L^{2}(\Omega)}+\left(\frac{1}{\tau \sigma}+\frac{\sigma}{2}\right)\rho\|\delta_{u,h_{k}}^{k}\|_{L^{2}(\Omega)}+{\tilde C_{1}}h_{k+1}+\left(\frac{1}{\tau \sigma}+\frac{\sigma}{2}\right)\rho\|\delta_{u,h_{k+1}}^{k+1}\|_{L^{2}(\Omega)}+{\tilde C_{1}}h_{k+2}\\
&\leq \cdots\\
&\leq\|\theta^{0}\|_{L^{2}(\Omega)}+\left(\frac{1}{\tau \sigma}+\frac{\sigma}{2}\right)\rho\sum_{i=0}^{k}\|\delta_{u,h_{i+1}}^{i+1}\|_{L^{2}(\Omega)}+{\tilde C_{1}}\sum_{i=0}^{k}h_{i+2}\\
&<C'_{1}.
\end{aligned}
\end{equation}
By Lemma \ref{lmm:complexity2}, we have
\begin{equation}
\label{es2}
\|\bar \theta^{k+1}\|_{L^{2}(\Omega)}^{2}\leq \|\theta^{k}\|_{L^{2}(\Omega)}^{2}+ \bar \eta \|\delta^{k}_{u,h_{k+1}}\|_{L^{2}(\Omega)}\leq (C'_{1})^{2}+\bar \eta \|\delta^{k+1}_{u,h_{k+1}}\|_{L^{2}(\Omega)},
\end{equation}
so there exists a constant $C'_{2}$ such that
$\|\bar \theta^{k+1}\|_{L^{2}(\Omega)}\leq C'_{2}.$
Hence,
\begin{equation}
\left\|\theta^{k+1}+\bar \theta^{k+1}\right\|_{L^{2}(\Omega)}\leq \left\|\theta^{k+1}\right\|_{L^{2}(\Omega)}+\left\|\bar \theta^{k+1}\right\|_{L^{2}(\Omega)}<C'_{1}+C'_{2}.
\end{equation}
By the definition of $\theta_{k}$ and $\bar \theta_{k}$, we have
\begin{equation}
\| \theta^{k+1}-\bar \theta^{k+1}\|_{L^{2}(\Omega)}^{2}=
\|\left(\frac{1}{\sqrt{2 \tau \sigma}}\left(\lambda^{k+1}_{h_{k+1}}-\lambda^{k+1}_{h_{k+2}}\right), \sqrt\frac{\sigma}{2}\left(z^{k+1}_{h_{k+2}}-z^{k+1}_{h_{k+1}}\right)\right)\|_{L^{2}(\Omega)}^{2}=O(h_{k+2}^{2}),
\end{equation}
thus 
$\| \theta^{k+1}-\bar \theta^{k+1}\|_{L^{2}(\Omega)}=O(h_{k+2}).$

Moreover, we have the estimate that
\begin{equation}
\begin{aligned}
\left\langle\delta^{k+1}_{u,h_{k+1}}, u^{k+1}_{h_{k+1}}-u^{\ast}_{h_{k+1}}\right\rangle &\leq \left\langle\delta^{k+1}_{u,h_{k+1}}, u^{k+1}_{h_{k+1}}-z^{k+1}_{h_{k+1}}+z^{k+1}_{h_{k+1}}-z^{\ast}_{h_{k+1}}\right\rangle\\
&=\left\langle\delta^{k+1}_{u,h_{k+1}}, \frac{1}{\tau\sigma}(\lambda^{k+1}_{h_{k+1}}-\lambda^{\ast}_{h_{k+1}}+\lambda^{\ast}_{h_{k+1}}-\lambda^{k}_{h_{k+1}})\right\rangle-\left\langle\delta^{k+1}_{u,h_{k+1}}, z^{k+1}_{h_{k+1}}-z^{\ast}_{h_{k+1}}\right\rangle\\
&\leq \frac{1}{\tau\sigma}\|\delta^{k+1}_{u,h_{k+1}}\|_{L^{2}(\Omega)}\left(\|\lambda^{k+1}_{h_{k+1}}-\lambda^{\ast}_{h_{k+1}}\|_{L^{2}(\Omega)}+\|\lambda^{\ast}_{h_{k+1}}-\lambda^{k}_{h_{k+1}}\|_{L^{2}(\Omega)}+\|z^{k+1}_{h_{k+1}}-z^{\ast}_{h_{k+1}}\|_{L^{2}(\Omega)}\right)\\
&\leq \bar \eta \|\delta^{k+1}_{u,h_{k+1}}\|_{L^{2}(\Omega)}.
\end{aligned}
\end{equation}
where $\bar \eta:=\sqrt{\frac{2C_{1}'}{\tau\sigma}}+\sqrt{\left(1+\frac{1}{\tau}\right)\frac{2C_{2}'}{\sigma}}$ is a constant, we used (\ref{es1}), (\ref{es2}) and the property $u^{\ast}_{h_{k+1}}=z^{\ast}_{h_{k+1}}$.

Then we know from Lemma \ref{lmm:complexity2} that 
\begin{equation}
\begin{aligned}
&\sum_{k=0}^{\infty}\left(\frac{\sigma}{2}\|u^{k+1}_{h_{k+1}}-z^{k+1}_{h_{k+1}}\|_{L^{2}(\Omega)}^{2}+\frac{\sigma}{2}\|u^{k+1}_{h_{k+1}}-z^{k}_{h_{k+1}}\|_{L^{2}(\Omega)}^{2}\right)\\
\leq& \sum_{k=0}^{\infty}\left(\left\|\theta^{k}\right\|^{2}_{L^{2}(\Omega)}-\left\|\theta^{k+1}\right\|^{2}_{L^{2}(\Omega)}+\left\|\theta^{k+1}\right\|^{2}_{L^{2}(\Omega)}-\left\|\bar \theta^{k+1}\right\|^{2}_{L^{2}(\Omega)}\right)+\sum_{k=0}^{\infty}\left\langle\delta^{k+1}_{u,h_{k+1}}, u^{k+1}_{h_{k+1}}-u^{\ast}_{h_{k+1}}\right\rangle
\\
\leq & \left\|\theta^{0}\right\|^{2}_{L^{2}(\Omega)}+\bar \eta \sum_{k=0}^{\infty}\|\delta^{k+1}_{u,h_{k+1}}\|_{L^{2}(\Omega)}+\sum_{k=0}^{\infty}(C'_{1}+C'_{2})\cdot{\rm max}_{k}(\|\bar \theta^{k+1}-\theta^{k+1}\|_{L^{2}(\Omega)})
\\
< &\infty,
\end{aligned}
\end{equation}
where we used the property $\left\|\theta^{k+1}\right\|^{2}_{L^{2}(\Omega)}-\left\|\bar \theta^{k+1}\right\|^{2}_{L^{2}(\Omega)}\leq \|\bar \theta^{k+1}+\theta^{k+1}\|_{L^{2}(\Omega)}\cdot\|\bar \theta^{k+1}-\theta^{k+1}\|_{L^{2}(\Omega)}$ and (\ref{theta}).
Hence, there exist constants $\tilde C_{1}, \tilde C_{2}$ such that
\begin{equation}\label{equ:bounded}
\sum_{k=0}^{\infty}\|u^{k+1}_{h_{k+1}}-z^{k+1}_{h_{k+1}}\|_{L^{2}(\Omega)}^{2}\leq \tilde C_{1},\ \ \sum_{k=0}^{\infty}\|u^{k+1}_{h_{k+1}}-z^{k}_{h_{k+1}}\|_{L^{2}(\Omega)}^{2}\leq \tilde C_{2}.
\end{equation}

Finally, by substituting (\ref{equ:bounded}) to (\ref{R inequality}), there exists a constant $\tilde C$,
\begin{equation}
\sum_{k=0}^{\infty}R_{h_{k+1}}(u^{k+1}_{h_{k+1}},z^{k+1}_{h_{k+1}},\lambda^{k+1}_{h_{k+1}})\leq 2
\sum_{k=0}^{\infty}\|\delta_{u,h_{k+1}}^{k+1}\|_{L^{2}(\Omega)}^{2}+2\sigma^{2}\tilde C_{1}+(1+\sigma^{2})\tilde C_{2}=\tilde C<\infty.
\end{equation}
Thus, by Lemma \ref{lmm:complexity1}, we know that
\begin{equation}
\min_{1\leq i\leq k}{R_{h_{i}}}(u^{i}_{h_{i}}, z^{i}_{h_{i}}, \lambda^{i}_{h_{i}})\leq \frac{\tilde C}{k}, \ \ \lim \limits_{k\rightarrow \infty}(k\times \min_{1\leq i\leq k}{R_{h_{i}}}(u^{i}_{h_{i}}, z^{i}_{h_{i}}, \lambda^{i}_{h_{i}}))=0
\end{equation}
holds. 
Therefore, combining the obtained global convergence results, we complete the whole proof of Theorem \ref{convergence theorem}.
\end{proof}

\section{Numerical experiments}
\label{sec:4}
In this section we illustrate the numerical performance of the proposed multi-level ADMM algorithm for PDE-constrained optimization problems. All our computational results are obtained by MATLAB R2017b running on a computer with 64-bit Windows 7.0 operation system, Intel(R) Core(TM) i7-6700U CPU (3.40 GHz), and 32 GB of memory.

First, we introduce the algorithmic details that are common to all examples.
The discretization was carried out by using the standard piecewise linear finite element approach.
To present numerical results, it is convenient to introduce the experimental order of convergence (EOC), which for some positive error function $E(h):=\|u-u_{h}\|_{L^{2}(\Omega)}, h>0$ is defined by
\begin{equation}
{\rm EOC}:=\frac{{\rm log}E(h_{1})-{\rm log}E(h_{2})}{{\rm log}h_{1}-{\rm log}h_{2}}.
\end{equation}
We note that if $E(h)=O(h^{\beta})$, then $EOC\approx \beta$. 
In numerical experiments, we measure the accuracy of an approximate optimal solution by using the corresponding KKT residual error for each algorithm. For the purpose of showing the efficiency of our mADMM, we report the numerical results obtained by running the ihADMM (see \cite{Song2017Fe} for details) and the classical ADMM method to compare with the results obtained by the mADMM. In this case, we terminate all the algorithms when $\eta < 10^{-6}$ with the maximum number of iterations set to 500. For all numerical examples and all algorithms, we choose zeros as the initial values and the penalty parameter $\sigma$ was chosen as $\sigma=0.1\alpha$. About the step length $\tau$, we choose $\tau = 1.618$. 

\begin{example}
\label{ex1}
(\cite{Hinze2009Optimization}, Example 3.3)
Consider
\begin{equation*}
\begin{aligned}
\min \limits_{(y,u)\in H_{0}^{1}(\Omega)\times L^{2}(\Omega)}^{}\ \ J(y,u)&=\frac{1}{2}\|y-y_d\|_{L^2(\Omega)}^{2}+\frac{\alpha}{2}\|u\|_{L^2(\Omega)}^{2} \\
{\rm s.t.}\qquad \quad \quad \ -\Delta y&=u\ \ \ \mathrm{in}\  \Omega, \\
\ \ \ \ \ \ y&=0\quad  \mathrm{on}\ \partial\Omega,\\
\qquad \qquad\quad u&\in  U_{ad}=\{v(x)|a\leq v(x)\leq b, {\rm a.e }\  \mathrm{on}\ \Omega\},
\end{aligned}
\end{equation*}
where the domain is the unit circle $\Omega=B_{1}(0)\subseteq \mathbb{R}^{2}$. Set the desired state $y_{d}=(1-(x_{1}^{2}+x_{2}^{2}))x_{1}$, the parameters $\alpha=0.1, a=-0.2, b=0.2$. 
\end{example}
In this example, the exact solutions of the problem are unknown in advance. Instead we use the numerical solutions computed on the grid with $h=2^{-10}$ as reference solutions.
As an example, the discretized optimal control on the grid with $h=2^{-7}$ is presented in Figure \ref{fig:1}.

\begin{figure}[H]
\centering
\includegraphics[scale=0.35]{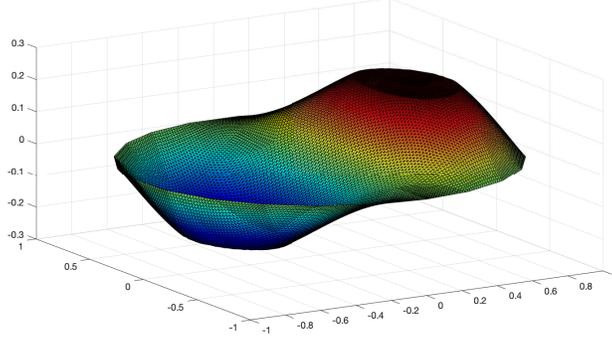}
\caption{Discretized optimal control solution for Example \ref{ex1} on the grid with $h = 2^{-7}$.}
\label{fig:1}       
\end{figure}
The error of the control $u$ {\rm w.r.t.} the $L^{2}$-norm, the EOC for the control, the numerical results for the accuracy of solution, the CPU time and the number of iterations obtained by our mADMM, the ihADMM and the classical ADMM are shown in Table \ref{tab:1}. We can see from Table \ref{tab:1} that our mADMM is highly efficient in obtaining an approximate solution compared with the ihADMM and the classical ADMM in terms of the CPU time, especially when the discretization is in a fine level.  Furthermore, it should be specially mentioned that the numerical results in terms of iterations illustrate the mesh-independent performance of the mADMM and the ihADMM. However, iterations of the classical ADMM will increase with the refinement of the discretization.

\begin{table}[H]
\centering
\caption{The convergence behavior of our mADMM, the ihADMM and the classical ADMM for Example \ref{ex1}.}
\label{tab:1}       
\begin{tabular}{lllllllll}
\hline\noalign{\smallskip}
 h & $\#\rm dofs$ & $E$ &EOC & Index  & mADMM & ihADMM & classical ADMM\\
\noalign{\smallskip}\hline\noalign{\smallskip}
 $2^{-5}$ & 635 &0.00168 &  - & residual  $\eta$& 8.46e-07 & 9.29e-07 & 8.32e-07\\
  & & & & CPU times/s & 0.21 & 0.61 & 0.95\\
  & & & & $\#$iter & 14 & 27 & 63\\
\noalign{\smallskip}\hline\noalign{\smallskip}
 $2^{-6}$ & 2629 &5.57e-04 & 1.5927 & residual  $\eta$& 6.90e-07 & 8.78e-07 & 2.47e-07\\
 &  & & & CPU times/s & 0.43 & 1.73 & 2.13\\
 & & & & $\#$iter & 15 & 26 & 30\\
\noalign{\smallskip}\hline\noalign{\smallskip}
 $2^{-7}$ & 10697 &2.05e-04 & 1.4420& residual  $\eta$& 7.14e-07 & 8.83e-07 & 9.03e-07\\
 & & & & CPU times/s & 1.23 & 7.40 & 30.46\\
 & & & & $\#$iter & 13 & 25 & 54\\
\noalign{\smallskip}\hline\noalign{\smallskip}
 $2^{-8}$ & 43153 &1.13e-04 & 0.8593& residual  $\eta$& 4.29e-07 & 7.16e-07 & 6.93e-07\\
 & & & & CPU times/s & 4.45 & 43.13 & 755.84\\
 & & & & $\#$iter & 13 & 25 & 119\\
\noalign{\smallskip}\hline\noalign{\smallskip}
 $2^{-9}$ & 173345 &6.68e-05 & 0.7584& residual  $\eta$& 1.22e-07 & 8.69e-07 & 2.68e-07\\
 & & & & CPU times/s & 39.15 & 384.39 & 39646.84\\
 & & & & $\#$iter & 14 & 24 & 380\\
\noalign{\smallskip}\hline\noalign{\smallskip}
 $2^{-10}$ & 694849 & - & - & residual  $\eta$& 8.7e-08 & 9.29e-07 & \textbf{3.58e-05}\\
  & & & & CPU times/s & 553.37 & 6451.98 & 289753.72\\
  & & & & $\#$iter & 15 & 23 & \textbf{500}\\
\noalign{\smallskip}\hline
\end{tabular}
\end{table}

\begin{example}
\label{ex2}
(\cite{Hinze2012The}, Example 4.1)
Consider
\begin{equation*}
\begin{aligned}
\min \limits_{(y,u)\in H_{0}^{1}(\Omega)\times L^{2}(\Omega)}^{}\ \ J(y,u)&=\frac{1}{2}\|y-y_d\|_{L^2(\Omega)}^{2}+\frac{\alpha}{2}\|u\|_{L^2(\Omega)}^{2} \\
{\rm s.t.}\qquad \quad \quad \  -\Delta y&=u\ \ \ \mathrm{in}\  \Omega, \\
y&=0\quad  \mathrm{on}\ \partial\Omega,\\
\qquad \qquad\quad u&\in  U_{ad}=\{v(x)|a\leq v(x)\leq b, {\rm a.e }\  \mathrm{on}\ \Omega\}.
\end{aligned}
\end{equation*}
where $\Omega=(0,1)^{2}$, the upper bound is $a=0.3$, the lower bound is $b=1$, and the regularization parameter is $\alpha=0.001$. We choose $y_{d}=-4\pi^{2} \alpha \sin(\pi x)\sin(\pi y)+Sr$, where $r=\min(1,\max(0.3,2\sin(\pi x)\sin(\pi y)))$, $S$ denotes the solution operator. In addition, from the choice of parameters, it is easy to know that $u=r$ is the unique solution of the continuous problem. 
\end{example}
The exact control and the discretized optimal control on the grid with $h = \sqrt{2}/2^{7}$ are presented in Figure \ref{fig:2}. 
\begin{figure}[H]
\centering
	\begin{subfigure}[t]{3in}
		\centering
		\includegraphics[width=3.4in]{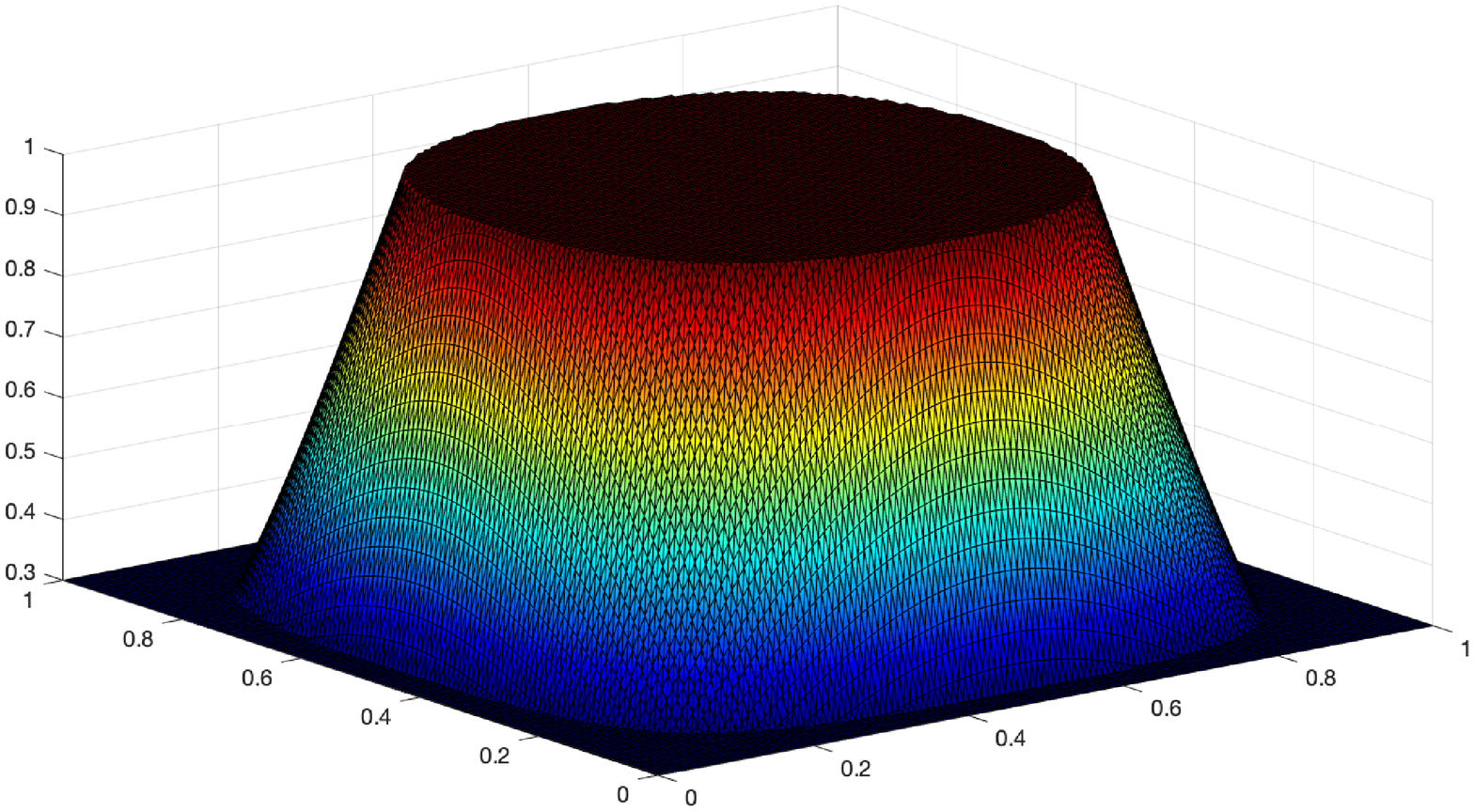}
		\caption{exact control $u$}\label{fig:1a}		
	\end{subfigure}
	\qquad
	\begin{subfigure}[t]{3in}
		\centering
		\includegraphics[width=3.4in]{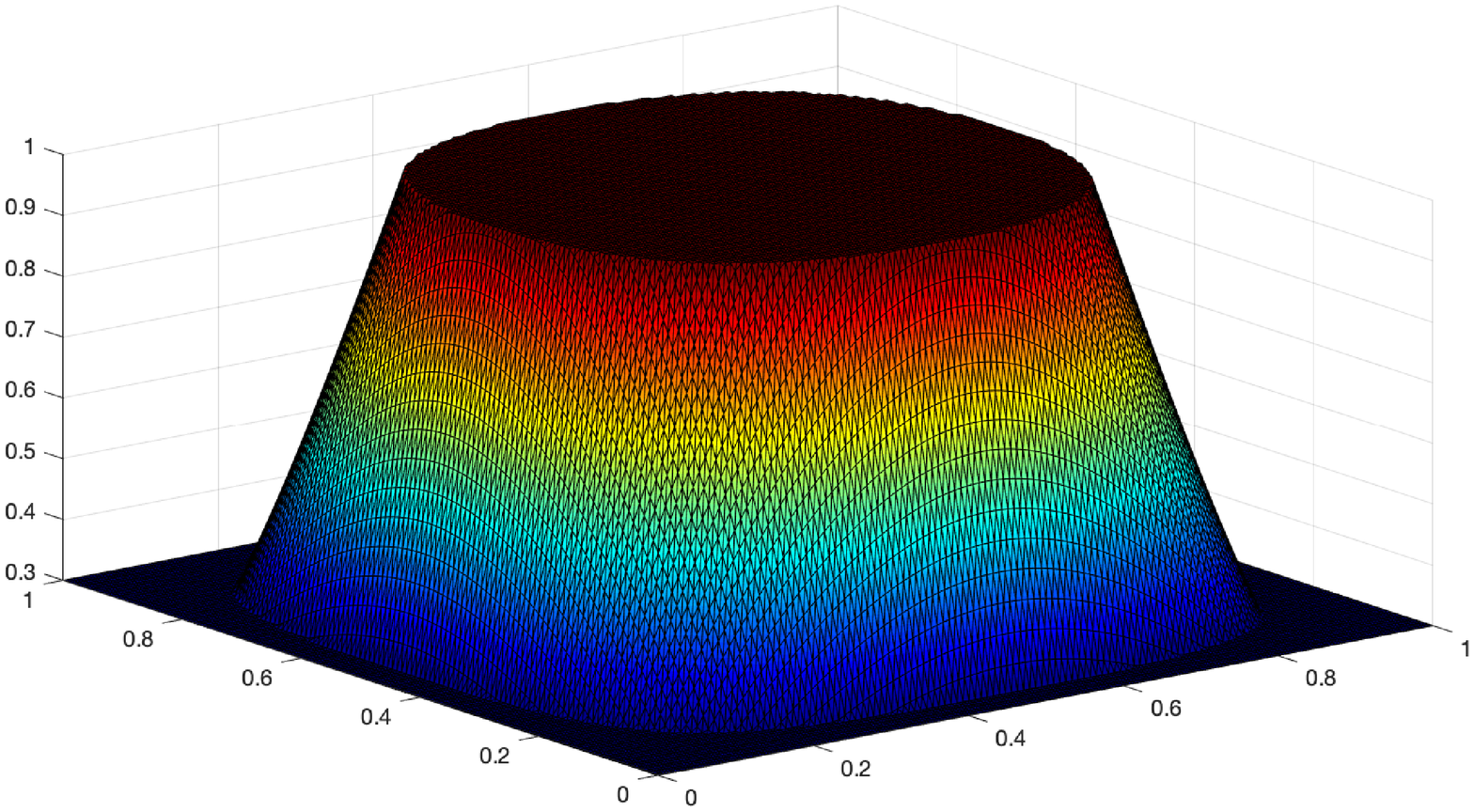}
		\caption{numerical control $u_{h}$}\label{fig:1b}
	\end{subfigure}
	\caption{Exact control solution (left) and discretized optimal control solution (right) for Example \ref{ex2} on the grid with $h = \sqrt{2}/2^{7}$.}\label{fig:2}
\end{figure}
The error of the control $u$ {\rm w.r.t.} the $L^{2}$- norm, the EOC for the control, the numerical results for the accuracy of solution, the CPU time and the number of iterations obtained by our mADMM, the ihADMM and the classical ADMM are shown in Table \ref{tab:2}. Experiment results show that our mADMM has evident advantage on CPU time over the ihADMM and the classical ADMM. Furthermore, we also notice that the numerical results in terms of iteration numbers illustrate the mesh-independent performance of the mADMM. 
\begin{table}[H]
\centering
\caption{The convergence behavior of our mADMM, the ihADMM and the classical ADMM for Example \ref{ex2}.}
\label{tab:2}       
\begin{tabular}{lllllllll}
\hline\noalign{\smallskip}
 h & $\#\rm dofs$ & $E$ &EOC & Index  & mADMM & ihADMM & classical ADMM\\
\noalign{\smallskip}\hline\noalign{\smallskip}
 $\sqrt{2}/2^{4}$ & 225 &0.0172 &- & residual  $\eta$& 9.26e-07 & 9.88e-07 & 9.98e-07\\
 && & & CPU times/s & 0.31 & 0.28 & 0.43\\
 & & & & $\#$iter & 22 & 25 & 120\\
\noalign{\smallskip}\hline\noalign{\smallskip}
 $\sqrt{2}/2^{5}$ & 961 &6.71e-03 &1.3580 & residual  $\eta$& 7.83e-07 & 9.19e-07 & 7.25e-07\\
 & & & & CPU times/s & 0.85 & 0.69 & 0.87\\
 & & & & $\#$iter & 23 & 26 & 32\\
\noalign{\smallskip}\hline\noalign{\smallskip}
 $\sqrt{2}/2^{6}$ & 3969 &2.11e-03 &1.6691& residual  $\eta$& 8.62e-07 & 7.10e-07 & 3.34e-07\\
 & & & & CPU times/s & 2.97 & 3.85 & 4.45\\
 & & & & $\#$iter & 24 & 30 & 32\\
\noalign{\smallskip}\hline\noalign{\smallskip}
 $\sqrt{2}/2^{7}$ & 16129 &8.02e-04 &1.3956& residual  $\eta$& 9.80e-08 & 7.80e-08 & 8.86e-08\\
 && & & CPU times/s & 14.92 & 37.39 & 91.15\\
 & & & & $\#$iter & 23 & 28 & 78\\
\noalign{\smallskip}\hline\noalign{\smallskip}
 $\sqrt{2}/2^{8}$ & 65025 &3.58e-04 &1.1636& residual  $\eta$& 9.04e-07 & 9.62e-07 & 7.05e-07\\
 & & & & CPU times/s & 22.80 & 151.57 & 2457.02\\
 & & & & $\#$iter & 21 & 28 & 183\\
  \noalign{\smallskip}\hline\noalign{\smallskip}
 $\sqrt{2}/2^{9}$ & 261121 &1.81e-04 &0.9840 & residual  $\eta$& 3.09e-07 & 8.43e-07 & 5.56e-07\\
 && & & CPU times/s & 168.58 & 1469.13 & 40739.35\\
 & & & & $\#$iter & 22 & 30 & 283\\
\noalign{\smallskip}\hline
\end{tabular}
\end{table}

\section{Conclusion}
\label{sec:5}
In this paper, we employ a multi-level ADMM algorithm to solve optimization problems with PDE constraints. Instead of solving the discretized problems, we apply the \emph{`optimize-discretize-optimize'} strategy. Such approach has the flexibility that allows us to discretize the subproblems of the inexact ADMM algorithm by different discretization schemes. Motivated by the multi-level strategy, we propose the proper strategy of gradually refining the grid and the strategy of solving the subproblems inexactly. We designed the convergent multi-level ADMM (mADMM) algorithm, which can significantly reduce the computation cost and make the algorithm faster. The convergence analysis and the iteration complexity results $o(1/k)$ is presented. Numerical results demonstrated the efficiency of the proposed mADMM algorithm. 

\section*{Acknowledgements}
We would like to thank Prof. Long Chen very much for his FEM package iFEM \cite{Chen2009} in Matlab.

\end{document}